\documentclass[12pt, reqno]{amsart}

\usepackage[margin=0.75in]{geometry}
\usepackage[
  bookmarks=true]{hyperref}
\usepackage[OT2, T1]{fontenc}
\usepackage[english]{babel}
\usepackage[utf8]{inputenc}
\usepackage{csquotes}
\usepackage[final]{microtype}
\usepackage{lmodern}
\usepackage{amsthm}
\usepackage{amssymb}
\usepackage{mathrsfs}
\usepackage{enumerate}
\usepackage{tikz-cd} 
\usepackage{tikz}
\usepackage{multicol}
\usepackage{multirow}
\usepackage{tikz-qtree}
\usepackage{rotating}
\usepackage{graphicx}
\usepackage{comment}
\usepackage{enumitem}
\usepackage{array}
\usetikzlibrary{arrows,calc,matrix,trees,arrows.meta,positioning,decorations.pathreplacing,bending}

\newtheorem{theorem}{Theorem}[section]

\newtheorem{proposition}[theorem]{Proposition}
\newtheorem{lemma}[theorem]{Lemma}
\newtheorem{corollary}[theorem]{Corollary}

\theoremstyle{definition}

\newtheorem{remark}[theorem]{Remark}
\newtheorem{definition}[theorem]{Definition}

\newtheorem{question}[theorem]{Question}
\newtheorem*{theorem*}{Theorem}
\newtheorem*{conjecture*}{Conjecture}

\newcommand{\Z}{\mathbb{Z}}
\newcommand{\F}{\mathbb{F}}

\newcommand{\SL}{\mathrm{SL}}
\newcommand{\Q}{\mathbb{Q}}
\newcommand{\Oh}{\mathcal{O}} 
\newcommand{\Sel}{\mathrm{Sel}} 
\newcommand{\Gal}{\mathrm{Gal}} 
\newcommand{\Conf}{\mathrm{Conf}}
\newcommand{\PConf}{\mathrm{PConf}}
\newcommand{\Frob}{\mathrm{Frob}}

\newcommand{\et}{\text{\'et}}

\newcommand{\genlegendre}[4]{%
  \genfrac{(}{)}{}{#1}{#3}{#4}%
  \if\relax\detokenize{#2}\relax\else_{\!#2}\fi
}
\newcommand{\legendre}[3][]{\genlegendre{}{#1}{#2}{#3}}

\newcolumntype{P}[1]{>{\centering\arraybackslash}p{#1}}

\DeclareSymbolFont{cyrletters}{OT2}{wncyr}{m}{n}
\DeclareMathSymbol{\Sha}{\mathalpha}{cyrletters}{"58}

\begin{document}

\title[Prob. dist. $\pi$-Selmer ranks of cyc. ord. $p$-Twists of $E/\F_q(t)$]{On the prime Selmer ranks of cyclic prime twist families of elliptic curves over global function fields}
\author{Sun Woo Park}
\email{spark483@wisc.edu, s.park@mpim-bonn.mpg.de}
\address{480 Lincoln Drive, Madison, WI, USA 53706 \\ Vivatsgasse 7, 53111 Bonn, Germany}

\keywords{Elliptic curves, Weil restrictions, Selmer groups, Global function fields, Markov processes}

\begin{abstract}
Fix a prime number $p$. Let $\F_q$ be a finite field of characteristic coprime to 2, 3, and contains the primitive $p$-th root of unity $\mu_p$. Based on the works by Swinnerton-Dyer and Klagsbrun, Mazur, and Rubin, we prove that the probability distribution of the sizes of prime Selmer groups over a family of cyclic prime twists of non-isotrivial elliptic curves over $\F_q(t)$ satisfying a number of mild constraints conforms to the distribution conjectured by Poonen and Rains with explicit error bounds. The key tools used in proving these results are the Riemann hypothesis over global function fields, the Erd\"os-Kac theorem, and the geometric ergodicity of Markov chains.
\end{abstract}

\maketitle

\vspace*{6pt}  
\tableofcontents

\newpage 

\section{Introduction} \label{section:introduction}

Let $p$ be a fixed prime number. Let $\mu_p$ be the set of primitive $p$-th roots of unity. We fix an element $\zeta_p$ which generates $\mu_p$. Let $K$ be the global function field $\mathbb{F}_q(t)$ of characteristic coprime to $2$ and $3$ which contains $\mu_p$, i.e. $q \equiv 1 \text{ mod } p$. Let $F_n(\F_q)$ be the set of monic polynomials of degree $n$ over $\F_q$.

Given a polynomial $f \in F_n(\F_q)$, there is a cyclic order-$p$ Galois extension $L^f := K(\sqrt[p]{f})$ over $K$. Choose a generator $\sigma_{f}$ of the cyclic Galois group $\text{Gal}(K(\sqrt[p]{f})/K) \cong \mathbb{Z}/p\mathbb{Z}$. We may associate the field $L_f$ with a cyclic order-$p$ character $\chi_f \in \text{Hom}(\text{Gal}(\overline{K}/K),\mu_p)$ defined via the quotient map
\begin{equation*}
    \chi_f: \text{Gal}(\overline{K}/K) \twoheadrightarrow \text{Gal}(L^f/K) \to \mu_p
\end{equation*}
that maps $\sigma_f$ to $\zeta_p \in \mu_p$. Note that $L_f$ is the fixed field of $\text{Ker}(\chi_f)$ in  $\overline{K}$.

Fix a non-isotrivial elliptic curve $E$ over $K$. The goal of this manuscript focuses on understanding the following question.
\begin{question} \label{question}
    Compute $\text{rank}_\mathbb{Z} E(L^f) - \text{rank}_{\mathbb{Z}} E(K)$ for any $f \in F_n(\F_q)$.
\end{question}

We study the question above by constructing what we call the cyclic order $p$ twist of $E$, as suggested in \cite{MR07}. Denote by $E^{\chi_f}$ the $p-1$ dimensional abelian variety over $K$ defined as
\begin{equation}
    E^{\chi_f} := \text{Ker} \left( \text{Nm}_{K}^{L_f}: \text{Res}_K^{L_f} E \to E \right)
\end{equation}
where $\text{Nm}_K^{L_f}$ is the field norm map, and $\text{Res}_K^{L_f} E$ is the Weil restriction of scalars of $E$ with respect to the Galois extension $L_f/K$. It follows that
\begin{equation}
    \text{rank}_\mathbb{Z} E^{\chi_f}(K) = \text{rank}_\mathbb{Z} E(L^f) - \text{rank}_{\mathbb{Z}} E(K).
\end{equation}
Mazur and Rubin showed that $1 - \sigma_f \in \text{End}(E^{\chi_f}/K)$, and that there exists a $\text{Gal}(\overline{K}/K)$-equivariant isomorphism $E^{\chi_f}[1-\sigma_f] \cong E[p]$, see for example \cite[Chapter 3, Proposition 4.1]{MR07}. For the rest of the manuscript we use the abbreviation $\pi := 1-\sigma_f$, as stated in \cite[Chapter 6]{KMR14}. In particular, if $p = 2$, then $\pi = 2$, and $E^{\chi_f}$ is the quadratic twist of $E$ by the quadratic character $\chi_f$. 

One way to understand Question \ref{question} is by computing the $\pi$-Selmer group of the abelian variety $E^{\chi_f}$ over $K$. We recall that given a non-isotrivial abelian variety $A/K$ and $m \in \text{End}(A/K)$ an isogeny of $A$ of degree coprime to characteristic of $K$, the short exact sequence of group schemes
\begin{equation*}
    0 \to A[m] \to A \xrightarrow{m} A \to 0
\end{equation*}
induces the following commutative diagram,
\begin{center}
\begin{tikzcd}
    0 \arrow[r] & A(K)/mA(K) \arrow[r] \arrow[d] & H^1_{\et}(K, A[m]) \arrow[r] \arrow[d] & H^1_{\et}(K, A)[m] \arrow[r] \arrow[d] & 0 \\
    0 \arrow[r] & \prod_v A(K_v)/mA(K_v) \arrow[r] & \prod_v H^1_{\et}(K_v, A[m]) \arrow[r] & \prod_v H^1_{\et}(K_v, A)[m] \arrow[r] & 0,
\end{tikzcd}
\end{center}
where $v$ varies over all places of $K$. The $m$-Selmer group of the abelian variety $A$ is given by
\begin{equation} \label{defn:selmer}
    \Sel_m(A) := \text{Ker} \left( H^1_{\et}(K, A[m]) \to \prod_v H^1_{\et}(K_v, A)[m] \right).
\end{equation}

Given a universal family of elliptic curves over a global field $K$, Poonen and Rains made a conjecture on the distribution of $p$-Selmer groups of elliptic curves for some prime number $p$.
\begin{conjecture*}{\cite{PR12}}
    Let $K$ be a global field of characteristic coprime to $2$ and $3$. Let $p$ be a prime number coprime to the characteristic of $K$. Then as $E$ varies over all elliptic curves over $K$, 
    \begin{equation*}
        \mathbb{P} \left[ \dim_{\F_p} \Sel_p(E) = d \right] = \left( \prod_{j \geq 0} (1 + p^{-j})^{-1} \right) \left( \prod_{j=1}^d \frac{p}{p^j-1}\right).
    \end{equation*}
    The average size of $\Sel_p E$ over all elliptic curves $E/K$ is $p+1$.
\end{conjecture*}
We elaborate on the statement of the conjecture. We first compute the probability distribution of Selmer groups over the set of finitely many elliptic curves $y^2 = x^3 + Ax + B$ whose coefficients $A,B \in K$ have bounded height $\mathcal{B}$. The conjecture states that the limit of the probability distribution obtained from letting $\mathcal{B}$ to grow arbitrarily large can be explicitly determined.

In this manuscript, we focus on computing the dimension of the following family of $\pi$-Selmer groups of $E^{\chi_f}$, defined as
\begin{equation}
    \Sel_\pi(E^{\chi_f}) := \text{Ker}\left( H^1_{\et}(K,E[p]) \to \prod_{v \text{ place of } K} H^1_{\et}(K_v,E^{\chi_f})[\pi] \right),
\end{equation}
where we use the $\text{Gal}(\overline{K}/K)$-equivariant isomorphism $E^{\chi_f}[\pi] \cong E[p]$ to identify $H^1_{\et}(K, E^{\chi_f}[\pi]) \cong H^1_{\et}(K, E[p])$. The main theorem of this paper confirms the Poonen-Rains heuristics for these families of $\pi$-Selmer groups of $E^{\chi_f}$. We use the following abbreviation to denote the probability distribution of dimensions of $\Sel_\pi(E^{\chi_f})$ ranging over $f \in F_n(\F_q)$.
\begin{equation}
    \mathbb{P} \left[ \dim_{\F_p} \Sel_\pi(E^{\chi_f}) = j \; | \; f \in F_n(\F_q) \right] := \frac{\#\{f \in F_n(\F_q) \; | \; \dim_{\F_p} \Sel_\pi(E^{\chi_f}) = j\}}{\# F_n(\F_q)}
\end{equation}

\begin{theorem}{Main Theorem.}\label{theorem:main_theorem}
Fix a prime number $p$. Let $K = \F_q(t)$ be a global function field whose characteristic is coprime to $2$,$3$, and $q \equiv 1 \text{ mod } p$. Let $E: y^2 = F(x) = x^3 + Ax + B$ be an elliptic curve over $K$ which satisfies the following conditions.
\begin{enumerate}
    \item $E$ is non-isotrivial.
    \item $E$ contains a place of split multiplicative reduction.
    \item The Galois group $\Gal(K(E[p])/K)$ is isomorphic to $\SL_2(\F_p)$.
\end{enumerate}
Let $\alpha(p)$ be a constant defined as
\begin{equation*}
    \alpha(p) := \sup_{0 < \rho < 1} \left( \min \left( \rho \log \rho + 1 - \rho, \; \; -\rho \log \gamma_p, -\rho \log \left( \frac{p}{p^2-1} \right) \right) \right),
\end{equation*}
where $0 < \gamma_p < 1$ is a constant depending on $p$ as defined in Corollary \ref{corollary:uniform_markov}. Then for any small enough $\delta > 0$, there exist sufficiently large $n$ and a fixed constant $A_{E,p,q} > 0$ that depends only on $E$, $p$, and $q$ such that 
\begin{equation*}
    \left| \mathbb{P} \left[ \dim_{\F_p} \Sel_\pi(E^{\chi_f}) = j \; | \; f \in F_n(\F_q) \right] - \left( \prod_{m \geq 0} \frac{1}{1 + p^{-m}} \right) \left( \prod_{m=1}^j \frac{p}{p^m-1}\right) \right| < \frac{A_{E,p,q}}{n^{\alpha(p) - \delta}}
\end{equation*}
\end{theorem}
We hence obtain that under certain mild conditions, the distribution of 2-Selmer ranks of quadratic twist families of non-isotrivial elliptic curves $E$ conforms to the Poonen-Rains conjecture over any global function field $K = \F_q(t)$. Numerical computations on Sage based on \cite[Theorem 1.1, Section 2.1]{Bax05} allow us to obtain non-optimal upper bounds for $\gamma_p$, see discussion following after Corollary \ref{corollary:uniform_markov} for further details. Under such conditions, non-optimal lower bounds for $\alpha(p)$ given some values of $p = 2, 3, 5, 7$ can be approximated as follows:
\begin{itemize}
    \item $\alpha(2) \sim 3.151407606 \cdot 10^{-4}$ where $\rho \sim 0.9749998600$.
    \item $\alpha(3) \sim 1.183774032 \cdot 10^{-4}$ where $\rho \sim 0.9846526712$. 
    \item $\alpha(5) \sim 5.681643158 \cdot 10^{-6}$ where $\rho \sim 0.9966309470$.
    \item $\alpha(7) \sim 5.825004132 \cdot 10^{-7}$ where $\rho \sim 0.9989208421$.
\end{itemize}

\begin{remark}
    The condition that $E$ is non-isotrivial further implies that condition (ii) in the statement of Theorem \ref{theorem:main_theorem} can be obtained after taking a finite separable extension of any global function field $K = \F_q(t)$ \cite[Proposition 3.4]{BLV09}.
\end{remark}

As a corollary, we are able to obtain a partial answer to Question \ref{question}. We would like to thank Douglas Ulmer for enlightening discussions, from which an error of the previous version of the corollary was discovered.
\begin{corollary}
    Assume the conditions and notations as in Theorem \ref{theorem:main_theorem}. We denote by
    \begin{equation*}
        \mathbb{P}\left[ \text{rank}_\mathbb{Z} E(L^f) - \text{rank}_\mathbb{Z} E(K) = j \; | \; f \in F_n(\F_q) \right] := \frac{\#\{f \in F_n(\F_q) \; | \; \text{rank}_\mathbb{Z} E(L^f) - \text{rank}_\mathbb{Z} E(K) = j\}}{\# F_n(\F_q)}
    \end{equation*}
    Then for any non-negative integer $j \geq 0$, we have
    \begin{equation*}
        \lim_{n \to \infty} \mathbb{P}\left[ \text{rank}_\mathbb{Z} E(L^f) - \text{rank}_\mathbb{Z} E(K) \leq (p-1) \cdot j \; | \; f \in F_n(\F_q) \right] \leq \sum_{J = 0}^j \left( \prod_{m \geq 0} \frac{1}{1 + p^{-m}} \right) \left( \prod_{m=1}^J \frac{p}{p^m-1}\right)
    \end{equation*}
    In particular, for sufficiently large $p$, the rank of $E(L^f)$ increases by at most $p-1$ from the rank of $E(K)$ for almost all $f \in \mathbb{F}_q[t]$, and the rank of $E(L^f)$ is identical to that of $E(K)$ for at least approximately $50\%$ of $f \in \mathbb{F}_q[t]$.
\end{corollary}
\begin{proof}
    The corollary follows from the proof of \cite[Proposition 2.1, Proposition 6.3]{MR07}, where one uses the inequality $\text{corank}_{\mathbb{Z}_p[\sigma_f]} \Sel_{p^\infty}(E^{\chi_f}) \leq \dim_{\mathbb{F}_p} \Sel_{\pi}(E^{\chi_f}).$
\end{proof}

\begin{remark}
    We warn the readers, however, that the given upper bound is not binding for any values of $p \geq 3$ unlike the case for quadratic twist families of elliptic curves. This is because the $\pi$-torsion subgroup of the Tate-Shafarevich group of the abelian variety $E^{\chi_f}$ is not necessarily of an even dimension, as explicitly constructed by William Stein \cite{St02} and discussed in detail by Howe \cite{Ho01}. Specific conditions which can guarantee the Tate-shafarevich groups to be of even dimension are provided in \cite[Chapter 6]{MR07}. Indeed, there are conjectural statements by David, Fearnley, and Kisilevsky \cite{DFK07} and Mazur and Rubin \cite{MR19} who suggested that it is very unlikely that the ranks of the elliptic curves will increase by at least $1$ with respect to cyclic order-$p$ extensions over $\Q$. The function field analogue was carefully studied in a recent work by Comeau-Lapointe, David, Lalin, and Li \cite{CDLL22}, where they show that the conjecture fails for isotrivial cyclic twist families of elliptic curves, whereas numerical data suggests that the conjecture may hold for non-isotrivial cyclic twist families of elliptic curves.
\end{remark}

\section{Remarks and Outlines}

\subsection{Key Ingredients}

The three key ingredients utilized in proving the main theorem are as follows, all three of which contribute to the three terms for $\alpha(\rho)$ which determine the rate of convergence of the desired probability distribution to the Poonen-Rains distribution.
\begin{enumerate}
    \item \textbf{Effective Chebotarev Density Theorem} 
    \begin{itemize}
        \item \textbf{Relevant results}: Theorem \ref{theorem:effective_chebotarev}, Corollary \ref{corollary:effective_chebotarev}, Corollary \ref{corollary:quadratic_character_sum}, Proposition \ref{theorem:local_twists}
        \item \textbf{Error term}: $-\rho \log \left(\frac{p}{p^2-1} \right)$, arising from the density that the Frobenius element of a place of $K$ has order prime to $p$ inside $\text{Gal}(K(E[p])/K) \cong \text{SL}_2(\F_p)$.
    \end{itemize}
    \item \textbf{Effective Erd\"os-Kac Theorem} 
    \begin{itemize}
    \item \textbf{Relevant results}: Theorem \ref{theorem:effective_erdos_kac}, Proposition \ref{proposition:f_mu_irred}, Proposition \ref{proposition:fan_approximation}
    \item \textbf{Error term}: $\rho \log \rho + 1 - \rho$, arising from the probability that a degree $n$ polynomial has at least $\rho (\log n + \log \log q)$ and at most $2(\log n + \log \log q)$ many distinct irreducible factors.
    \end{itemize}
    \item \textbf{Geometric Convergence of Markov Chains} 
    \begin{itemize}
        \item \textbf{Relevant results}: Corollary \ref{corollary:uniform_markov}
        \item \textbf{Error term}: $-\rho \log \left(1 - \frac{p}{p^2-1}\right)$, arising from geometric rate of convergence of the constructed Markov chain to the stationary distribution.
    \end{itemize}
\end{enumerate}

\subsection{Outline of the proof}

We provide the outline of the proof of the main theorem along with the organization of this manuscript. We let $\rho$ to be a parameter whose value is between $0$ and $1$. The motivation for the proof originates from the previous work by Swinnerton-Dyer \cite{SD08} and Klagsbrun, Mazur and Rubin \cite{KMR14} who studied Lagrangian Markov operators over $\mathbb{Z}_{\geq 0}$ which govern the distribution of dimensions of $\pi$-Selmer groups over number fields.
\begin{enumerate}
    \item \textbf{Effective theorems}: In Section \ref{section:effective_theorems_RH}, we discuss the effective versions of Chebotarev density theorem and Erd\"os-Kac theorem used in the rest of the manuscript.
    \item \textbf{Finding a nice subset of polynomials}: Let $f \in F_n(\F_q)$. Suppose that $f$ admits a factorization $f = f_* f^*$, where $f^*$ is a product of irreducible factors of $f$ (including multiplicities) of degree greater than $\frac{4(\log n)^2}{\log q}$. In Section \ref{section:splitting}, we define the notion of splitting partitions and show using Merten's theorem and the effective Erd\"os-Kac theorem that for almost all $f \in F_n(\F_q)$ the following three conditions are satisfied:
    \begin{itemize}
        \item The number of distinct irreducible factors of $f$ is between $\rho(\log n + \log \log q)$ and $2(\log n + \log \log q)$.
        \item The number of distinct irreducible factors of $f^*$ is at least $(1-\epsilon)\rho(\log n + \log \log q)$ for small enough $\epsilon > 0$.
        \item There is an irreducible factor of $f^*$ whose Frobenius element in $\text{Gal}(K(E[p])/K) \cong \SL_2(\F_p)$ has order prime to $p$.
    \end{itemize}
    \item \textbf{Equidistribution}: In Section \ref{section:equidistribution}, we prove equidistribution of $l$-th power residue symbols associated to a fixed number of irreducible polynomials over $\F_q$.
    \item \textbf{Local Selmer groups}: In Section \ref{section:local_twists}, we recall the definition of local Selmer groups of $E$ associated to cyclic order $p$ local characters as shown in \cite{KMR14}. We use the ideas from \cite[Proposition 9.4]{KMR14} and the effective Chebotarev theorem to identify Chebotarev conditions that govern the image of the global cohomology group $H^1_{\et}(K,E[p])$ with respect to the localization map at a place $v$ of $K$.
    \item \textbf{Auxiliary Place}: In Section \ref{section:auxiliary_places}, we define the notion of the auxiliary place of $f$ satisfying the aforementioned three conditions, which is an irreducible factor of highest degree whose Frobenius element in $\text{Gal}(K(E[p])/K) \cong \SL_2(\F_p)$ has order prime to $p$. Using the equidistribution results from Section \ref{section:equidistribution} and the Chebotarev conditions from Section \ref{section:local_twists}, we construct a Markov operator defined over $\mathbb{Z}_{\geq 0}$ which governs the distribution of the dimensions of local Selmer groups of $E$ associated to cyclic order $p$ characters. This proves the effective version of the construction of governing Markov operators, as stated in \cite[Theorem 4.3, Theorem 9.5]{KMR14} and \cite[Theorem 1]{SD08}.
    \item \textbf{Lagrangian Markov operators}: In Section \ref{section:governing_markov_chain}, we analyze the stochastic properties of the governing Markov operator, such as its stationary distribution and effective rates of convergence.
    \item \textbf{Combining all ingredients}: In Section \ref{section:proofmain}, we prove the main theorem by approximating the desired probability distribution with the distribution of dimensions of local Selmer groups over the set of polynomials satisfying the three aforementioned conditions from Section \ref{section:polynomials}. Combined with the rate of convergence of the governing Markov operator from Section \ref{section:governing_markov_chain}, we prove that each ingredient gives rise to the rate of convergence of the desired probability distribution to the Poonen-Rains distribution.
\end{enumerate}

\subsection{Relevant works}
The statements of the Poonen-Rains conjecture are known for certain large families of elliptic curves, such as the universal family of elliptic curves ordered by height, or quadratic twist families of elliptic curves ordered by the norm of the twist. 

Suppose $K = \Q$. We list some previous studies which focused on computing the probability distribution of Selmer groups over certain families of elliptic curves.
\begin{itemize}
    \item Bhargava and Shankar compute the first moments of 2,3,4 and 5-Selmer groups over the universal family of elliptic curves, see for example \cite{BS15}.
    \item Heath-Brown, Swinnerton-Dyer, and Kane compute the probability distribution of 2-Selmer groups over the quadratic twist families of elliptic curves with full 2-torsions and no cyclic subgroup of order $4$ over $\Q$ \cite{HB94, SD08, Ka13}.
    \item Klagsbrun, Mazur, and Rubin generalized the construction of Markov chains suggested by Swinnerton-Dyer \cite{SD08} to compute the probability distribution of 2-Selmer groups over the quadratic twist families of elliptic curves with $\Gal(K(E[2])/K) = S_3$. Note that the elliptic curves are ordered in a non-canonical manner using Fan structures. They obtain the probability distribution of prime Selmer groups over non-canonically ordered cyclic order-$p$ twist families of elliptic curves with $\Gal(K(E[p])/K) = SL_2(\F_p)$ as well \cite{KMR14}.
    \item Smith successfully calculates the probability distribution of 2-Selmer groups over quadratic twist families of elliptic curves of bounded height $H$ except for some cases where $E[2](\Q) = \Z/2\Z$ or $\Z/2\Z \oplus \Z/2\Z$. As the upper bound on the height $H$ grows to infinity, the error bounds of the probability distribution is given by an order of $O(e^{-c (\log \log \log H)^{\frac{1}{4}}})$ for some constant $c > 0$. Smith utilizes Markov chains which govern the variations of kernel ranks of alternating square matrices whose entries are values of the Cassels-Tate pairings. Note that the Markov chains Smith utilized are different from those constructed by Swinnterton-Dyer and Klagsbrun, Mazur, and Rubin \cite{Sm17, Sm20, Sm22_01, Sm22_02}. 
    \item The Markov chains suggested by Smith can be utilized to prove the Cohen-Lenstra heuristics on $l^\infty$-torsion subgroups of class groups of cyclic $l$-extensions of $\Q$ (assuming the generalized Riemann hypothesis) \cite{KP21}, and Stevenhagen's conjecture on the asymptotic behavior of the solubility of negative Pell equations \cite{KP22}.
\end{itemize}

Consider the case where $K = \F_q(t)$ is of characteristic coprime to $2$ and $3$. Previous studies computed the probability distribution of $p$-Selmer groups of families of elliptic curves over global function fields $\F_q(t)$ under different conditions. Denote by $\mathcal{M}_n(\F_q)$ a finite subfamily of elliptic curves $E$ over $\F_q(t)$ of a fixed height $n$. The height of an elliptic curve is determined by the degrees of coefficient terms of $E$. (Of course, the choice of the height depends on over which families of elliptic curves the probability distribution of $p$-Selmer groups is computed.) 

Given a non-negative integer $j$, denote by $\mathbb{P} \left[ \dim_{\F_p} \Sel_p(E) = j \; | \; E \in \mathcal{M}_n(\F_q) \right]$ the probability that the dimensions of $p$-Selmer groups of finitely many elliptic curves of fixed height $n$ are equal to $j$. Below we list three probability distributions of $p$-Selmer groups of elliptic curves that can be computed over global function fields:
\begin{align}
    \lim_{n \to \infty} & \mathbb{P} \left[ \dim_{\F_p} \Sel_p(E) = j \; | \; E \in \mathcal{M}_n(\F_q) \right] \\
    \lim_{q \to \infty} \lim_{n \to \infty} & \mathbb{P} \left[ \dim_{\F_p} \Sel_p(E) = j \; | \; E \in \mathcal{M}_n(\F_q) \right]\\
    \lim_{n \to \infty} \lim_{q \to \infty} & \mathbb{P} \left[ \dim_{\F_p} \Sel_p(E) = j \; | \; E \in \mathcal{M}_n(\F_q) \right]
\end{align}
As before, we list some previous studies which focused on computing the desired probability distribution over $\mathbb{F}_q(t)$.
\begin{itemize}
    \item For the second limit (large-height, then large-$q$ limit), Ho, Le Hung, and Ngo \cite{HLHN14} compute the average size of 2-Selmer groups over the universal family of elliptic curves, whereas de Jong \cite{dJ02} computes that of 3-Selmer groups over the same family.
    \item Feng, Landesman, and Rains \cite{FLR20} prove that the third limit (large-$q$, then large-height limit) is equal to the Poonen-Rains distribution for any $m$-Selmer groups over universal families of elliptic curves, under the condition that $q$ is coprime to $2m$. They propose a Markov chain constructed from random kernel models, which governs the variation of $m$-Selmer groups over global function fields $\F_q(t)$. Using this Markov chain, they successfully prove the Poonen-Rains conjecture for $m$-Selmer groups of universal families of elliptic curves under the large $q$-limit.
    \item Landesman \cite{La21} demonstrates that the third limit of the average size of $m$-Selmer groups of universal families of elliptic curves conforms to the Poonen-Rains conjecture.
    \item The average size of $p$-Selmer groups of quadratic twist families of non-isotrivial elliptic curves under the third limit is computed by the author of this paper and Wang \cite{PW21}. 
    \item The key ingredient behind computing these distributions is a careful and rigorous determination of images of monodromy over algebraic spaces whose geometric fibers parametrize $p$-Selmer groups over a prescribed family of elliptic curves, see for instance \cite{dJF11, Ha06, EVW16}.
\end{itemize}
Theorem \ref{theorem:main_theorem} proves that the first limit (large-height limit) is equal to the Poonen-Rains distribution for $p = 2$ over quadratic twist families of elliptic curves.

\begin{remark}
We finally note that it is not always the case that the probability distribution of 2-Selmer groups over quadratic twist families of elliptic curves over a global field $K$ can be formulated. For example, Klagsbrun and Lemke Oliver showed that more than half the quadratic twists of elliptic curves over number fields $K$ with partial $K$-rational 2-torsion points (i.e. $E[2](K) = \Z/2\Z$) and without any cyclic 4-isogeny over $K$ have arbitrarily large $2$-Selmer ranks \cite{KL15}. Wang extends their results to global function fields $K = \F_q(t)$ in his Ph.D. thesis for arbitrary number of elements of the constant field $\F_q$ \cite{Wa21} .
\end{remark}

\section{Effective theorems from the Riemann hypothesis} \label{section:effective_theorems_RH}
We review some of the preliminary results on global function fields $K$ which will be utilized in computing the probability distribution of prime Selmer groups associated to cyclic prime twists of elliptic curves. Given a place $v$ over $K$, we denote by $\Frob_v$ the Frobenius element at $v$. Denote by $g_L$ the genus of a finite separable field extension $L/K$.

\subsection{Effective Chebotarev density theorem} \label{subsection:effective_chebotarev}
The effective version of Chebotarev density theorem over global function fields can be formulated as follows:
\begin{theorem}[Effective Chebotarev density theorem] \cite[Proposition 6.4.8]{FJ08}\label{theorem:effective_chebotarev}

Let $L/K$ be a Galois extension of global function fields over $\F_q(t)$. Pick a conjugacy class $C \subset G = \Gal(L/K)$. We use the variable $n$ to denote the degree of an irreducible polynomial $v$ of $\mathbb{F}_q[t]$. If the constant fields of $L$ and $K$ are both equal to $\F_q$, then
\begin{align*}
    & \left| \# \{ v \; \text{a place over} \; K \; | \; \Frob_v \in C, \; \dim_{\F_q} (\Oh_K/v) = n \} - \frac{|C|}{|G|} \frac{q^n}{n} \right| \\
    & < \frac{2|C|}{n|G|} \left[ (|G| + g_L) q^{\frac{n}{2}} + |G|(2g_K + 1)q^{\frac{n}{4}} + (|G| + g_L) \right].
\end{align*}
\end{theorem}
The constraint that the constant fields of $L$ and $K$ are identical allows one to reconstruct an analogue of the Chebotarev density theorem with explicit error bounds for function fields. Suppose the constant field of $L$, say $\F_{q^l}$, is a non-trivial extension of the constant field $\F_q$ of $K$. Then to compute the equation stated in Theorem \ref{theorem:effective_chebotarev}, one is required to compare whether the restriction of the conjugacy class $C$ to $\Gal(\F_{q^l}/\F_q)$ agrees with the $n$-th power of the arithmetic Frobenius $\tau:x \mapsto x^q$ as a cyclic generator of $\Gal(\F_{q^l}/\F_q)$. If not, then there are no places of degree $n$ whose Frobenius element lives inside the conjugacy class $C$. Note that the secondary error term is of $O(q^{\frac{n}{2}})$, which is obtained from the validity of the generalized Riemann hypothesis over $K = \F_q(t)$. For the analogous effective statements over number fields, see for example \cite{LO75}. We note that Galois extensions of global function fields with non-trivial constant field extensions also satisfy the following equation:
\begin{equation}
    \lim_{s \to 1^+} \frac{\sum_{\substack{v \; \text{a place over } K \\ \Frob_v \in C}} {|\{\Oh_K/v\}|^{-s}}}{\sum_{\substack{v \; \text{a place over } K}} {|\{\Oh_K/v\}|^{-s}}} = \frac{|C|}{|G|}
\end{equation}
where $s \to 1^+$ implies that $s$ approaches $1$ from above over the real values.

Using the explicit bounds obtained above, the density theorem can be obtained for any two conjugacy classes of the Galois group of the extension $L/K$ of function fields.
\begin{corollary}\label{corollary:effective_chebotarev}
Let $L/K$ be a Galois extension of global function fields over $\F_q(t)$. Pick two non-empty subsets $S, S' \subset G = \Gal(L/K)$ stable under conjugation. Suppose the following two conditions hold.
\begin{enumerate}
    \item The constant fields of $L$ and $K$ are both equal to $\F_q$.
    \item The size of the constant field $q$ satisfies
    \begin{equation*}
        q^{\frac{n}{2}} - q^{\frac{n}{4}} > 2(|G| + g_L + 2g_K)
    \end{equation*}
\end{enumerate}
We use the variable $n$ to denote the degree of an irreducible polynomial $v$ of $\mathbb{F}_q[t]$.
Then the following inequality holds.
\begin{align*}
    & \left| \frac{\{ v, \; \text{a place over} \; K \; | \; \Frob_v \in S, \; \dim_{\F_q} (\Oh_K/v) = n \}}{\{ v, \; \text{a place over} \; K \; | \; \Frob_v \in S', \; \dim_{\F_q} (\Oh_K/v) = n \}} - \frac{|S|}{|S'|} \right| \\
    & < 4\frac{|S|}{|S'|} (|G| + g_L + 2g_K) \left[ \frac{1}{q^\frac{n}{2} - q^\frac{n}{4} - 2(|G| + g_L + 2g_K) } \right].
\end{align*}
In particular, if $n \geq 2 \frac{\log 8 + \log(|G| + g_L + 2g_K)}{\log q}$, then
\begin{equation*}
    \left| \frac{\{ v, \; \text{a place over} \; K \; | \; \Frob_v \in S, \; \dim_{\F_q} (\Oh_K/v) = n \}}{\{ v, \; \text{a place over} \; K \; | \; \Frob_v \in S', \; \dim_{\F_q} (\Oh_K/v) = n \}} - \frac{|S|}{|S'|} \right| <  16 \frac{|S|}{|S'|} (|G| + g_L + 2g_K) q^{-\frac{n}{2}}.
\end{equation*}
\end{corollary}

\begin{remark}
We note that Deligne's proof of the Weil conjectures determines the error bounds of the effective Chebotarev density theorem. We refer to \cite[Theorem 9.13B]{Ro02} for further discussions.
\end{remark}

\subsection{Erd\"os-Kac Theorem} \label{subsection:erdos_kac}
Let $m$ be an integer. We denote by $w(m)$ the number of distinct irreducible factors of $m$. The Erd\"os-Kac Theorem states that the normal order of $w(m)$ is $\log \log m$. 
\begin{definition}
From this section and onwards, given two positive integers $n$ and $q \geq 5$, we denote by $m_{n,q}$ the quantity
\begin{equation}
    m_{n,q} := \log n + \log \log q
\end{equation}
\end{definition}
The Erd\"os-Kac Theorem over global function fields $K$ can be formulated as follows.
\begin{theorem}[Erd\"os-Kac Theorem for Function Fields] \cite[Theorem 1]{Li04} \label{theorem:erdos_kac}

Denote by $w(f)$ the number of distinct irreducible factors dividing a polynomial $f \in F_n(\F_q)$ of degree $n$. Then for any $a \in \mathbb{R}$,
\begin{equation*}
    \lim_{n \to \infty} \frac{\# \left\{ f \in F_n(\F_q) \; | \; \frac{w(f) - m_{n,q}}{\sqrt{m_{n,q}}}\leq a \right\}}{\# F_n(\F_q)} = \frac{1}{\sqrt{2 \pi}} \int_{-\infty}^{a} e^{-\frac{t^2}{2}}dt
\end{equation*}
\end{theorem}

Fix positive integers $\alpha, \beta$. We denote by
\begin{equation*}
    \mathbb{P} \left[ \alpha < w(f) < \beta \; \mid \; f \in F_n(\F_q) \right]
\end{equation*}
the probability that the number of irreducible factors of a square-free polynomial $f$ of degree $n$ over $\F_q$ is greater than $\alpha$ and less than $\beta$. In other words,
\begin{equation}
    \mathbb{P} \left[ \alpha \leq w(f) \leq \beta \; \mid \; f \in F_n(\F_q) \right] := \frac{\# \{ f \in F_n(\F_q) \; | \; \alpha \leq w(f) \leq \beta \}}{\#\{f \in F_n(\F_q) \}}
\end{equation}
Let $\rho$ be a positive number such that $0 < \rho < 1$. For sufficiently large $n$, the number of distinct prime divisors $w(f)$ for almost every polynomial $f \in F_n(\F_q)$ satisfies
\begin{equation*}
    \rho m_{n,q} \leq w(f) \leq 2 m_{n,q}.
\end{equation*}
The effective upper bound on the number of polynomials in $F_n(\F_q)$ that do not satisfy the condition above can be obtained as follows.

\begin{theorem}[Effective Erd\"os-Kac] \label{theorem:effective_erdos_kac}
For sufficiently large $n$, there exists a fixed constant $0 < C_{EK} < 4$ such that
\begin{equation}
    \mathbb{P} \left[ w(f) < \rho m_{n,q} \text{ or } w(f) > 2 m_{n,q} \; | \; f \in F_n(\mathbb{F}_q) \right] < C_{EK} (n \log q)^{-\rho \log \rho - 1 + \rho}.
\end{equation}
\end{theorem}
\begin{proof}
We thank the reviewer for suggesting the following idea of the proof. From \cite[Theorem 1]{FWY20}, we obtain that there exists a constant $0 < C_1 < 2$ such that
\begin{equation}
    \mathbb{P}[w(f) > 2 m_{n,q} \; | \; f \in F_n(\mathbb{F}_q)] < C_1 (n \log q)^{-2\log 2 - 1}.
\end{equation}
From \cite[Theorem 1]{FWY20} and \cite[Theorem 1]{Li04}, we also obtain that there exists a constant $0 < C_2 < 2$ such that
\begin{equation}
    \mathbb{P}[w(f) < \rho m_{n,q} \; | \; f \in F_n(\mathbb{F}_q)] < C_2 (n \log q)^{-\rho \log \rho + \rho - 1}.
\end{equation}
Combining two inequalities and the fact that for any $0 < \rho < 1$,
\begin{equation*}
    \rho \log \rho + 1 - \rho < 1 < 2 \log 2 + 1,
\end{equation*}
we obtain that there exists $0 < C_{EK} < 4$ such that
\begin{equation}
    \mathbb{P}[w(f) < \rho m_{n,q} \text{ or } w(f) > 2 m_{n,q} \; | \; f \in F_n(\mathbb{F}_q)] < C_{EK} (n \log q)^{-\rho \log \rho + \rho - 1}.
\end{equation}
\end{proof}

\begin{remark}
Theorem \ref{theorem:effective_erdos_kac} can also be obtained from using the results by Cohen, see for instance \cite[Theorem 6]{Co69} and \cite[Theorem 1.1]{CLNY22}.
\end{remark}

\section{Splitting partitions of polynomials} \label{section:polynomials}

The objective of this section is to find a suitable subset of polynomials in $F_n(\F_q)$ over which the behavior of $\Sel_\pi(E^{\chi_f})$ can be well understood. For this purpose, we introduce the notion of splitting partitions of polynomials. Our goal is to show that almost all $f \in F_n(\F_q)$ satisfies:
\begin{itemize}
        \item The number of distinct irreducible factors of $f$ is between $\rho m_{n,q}$ and $2 m_{n,q}$.
        \item The number of distinct irreducible factors of degree at least $\lfloor \frac{4m_{n,q}^2}{\log q} \rfloor$ is at least $(1-\epsilon)\rho m_{n,q}$ for some small enough $\epsilon > 0$.
        \item There is an irreducible factor of degree at least $\lfloor \frac{4m_{n,q}^2}{\log q} \rfloor$ whose Frobenius element in $\text{Gal}(K(E[p])/K) \cong \SL_2(\F_p)$ has order prime to $p$.
\end{itemize}

\subsection{Some sets of places}

\begin{definition}
    From this section and onwards, we assume the following conditions on $K = \F_q(t)$, prime $p$, and a fixed choice of an elliptic curve $E$ over $K$.
\begin{align} \label{equation:assumption_local_twists}
\begin{split}
    & \bullet E \text{ is a non-isotrivial elliptic curve over } K. \\ 
    & \bullet E \text{ has a place of split multiplicative reduction}. \\
    & \bullet \text{The constant field } \F_q \text{ has characteristic coprime to } 2,3,p, \text{ and contains } \mu_p. \\
    & \bullet \text{The image of } \Gal(\overline{K}/K) \to \text{Aut}(E[p]) \text{ contains } \SL_2(\mathbb{F}_p). \\
\end{split}
\end{align}
By Igusa's theorem, for any non-isotrivial elliptic curve $E$, there exists a prime $p$ and a finite separable extension of $K = \F_q(t)$ such that $E$ satisfies the first three conditions \cite{Ig59, BLV09}.
\end{definition}

\begin{definition} \label{defn:sets_places}
    The following notations are used to denote a set of places of $K$ whose definitions depend on the choice of the elliptic curve $E$. We follow the style of notations as stated in \cite[Section 3]{KMR14}.
    \begin{itemize}
    \item $\Sigma$: a set of places of $K$ that includes the places of bad reduction of $E$.
    \item $\Sigma_E$: the set whose elements are precisely the places of bad reduction of $E$.
    \item $\sigma$: a square-free product of places $v$ of $K$ such that $v \not\in \Sigma$.
    \item $\deg \sigma$: the sum of degrees of places $v \mid \sigma$, i.e. $\deg \sigma = \sum_{v \mid \sigma} \deg v$.
    \item $\Sigma(\sigma)$: a set of places of $K$ that includes a set of places in $\Sigma$ and a set of places dividing $\sigma$.
    \item $d_{\Sigma(\sigma)}$: the sum of degrees of elements in $\Sigma(\sigma)$, i.e. $d_{\Sigma(\sigma)} = \sum_{v \in \Sigma(\sigma)} \deg v$.
    \item For $0 \leq i \leq 2$, define the set
    \begin{equation*}
        \mathcal{P}_i := \{v \text{ place of } K \; | \; v \not\in \Sigma_E \; \text{and} \; \dim_{\mathbb{F}_p} E(K_v)[p] = i\}
    \end{equation*}
    The set $\mathcal{P}$ is the set
    \begin{equation*}
        \mathcal{P} := \{v \text{ place of } K \; | \; v \not\in \Sigma_E \} = \mathcal{P}_0 \cup \mathcal{P}_1 \cup \mathcal{P}_2.
    \end{equation*}
    Suppose in particular that $p = 2$. Given a Weierstrass equation of an elliptic curve $E: y^2 = F(x)$ satisfying the conditions from Theorem \ref{theorem:main_theorem}, denote by $L$ the cubic field extension $L = K[x]/(F(x))$. Note that the constant field of $L$ is equal to $\F_q$. The sets $\mathcal{P}_0, \mathcal{P}_1,$ and $\mathcal{P}_2$ correspond to sets of unramified places over $K$ not in $\Sigma$ which are inert, split into two places, or totally split in $L$.
    \item Given a positive number $d \in \mathbb{N}$, the set $\mathcal{P}_i(d)$ for $0 \leq i \leq 2$ is defined as
    \begin{equation*}
        \mathcal{P}_i(d) := \{ v \in \mathcal{P}_i \; \mid \; \deg v = d\}.
    \end{equation*}
    Likewise, the set $\mathcal{P}(d)$ is defined as
    \begin{equation*}
        \mathcal{P}(d) := \{ v \in \mathcal{P} \; \mid \; \deg v = d\}.
    \end{equation*}
    \end{itemize}
\end{definition}

Using the assumption (\ref{equation:assumption_local_twists}), we recall the following statement from \cite[Lemma 4.3]{KMR13} that the Frobenius elements of certain primes lying above a place $v$ over $K$ determine which classes of $\mathcal{P}_i$ the place $v$ lives in. Again, the original statement of the lemma is shown for arbitrary number fields, which can be extended to the case for global function fields.
\begin{lemma}{\cite[Lemma 4.3]{KMR13}} \label{lemma:prime_class}
Fix an elliptic curve $E/K$ satisfying the conditions stated in (\ref{equation:assumption_local_twists}). Let $v$ be a place over $K$ such that $v \not\in \Sigma$. Denote by $\Frob_v \in \Gal(K(E[p])/K)$ the Frobenius element associated to $v$. Then 
\begin{enumerate}
    \item $v \in \mathcal{P}_2 \; \iff \Frob_v = 1$ 
    \item $v \in \mathcal{P}_1 \; \iff \Frob_v \text{ has order exactly } p$
    \item $v \in \mathcal{P}_0 \; \iff \Frob_v^p \neq 1$
\end{enumerate}
\end{lemma}

\begin{remark} \label{remark:prime_split}
Igusa's theorem implies that any non-isotrivial elliptic curve satisfying conditions (\ref{equation:assumption_local_twists}) satisfies the condition that $\Gal(K(E[p])/K) \cong \SL_2(\F_p)$.

Denote by $g_{E[p]}$ the genus of the global function field $K(E[p])/K$. Computing the conjugacy classes of $\SL_2(\F_p)$ and Theorem \ref{theorem:effective_chebotarev} show that for sufficiently large $d$,
\begin{equation}
\label{equation:prime_split_density_general}
\max \left\{ \left| \frac{\# \mathcal{P}_0(d)}{\# \mathcal{P}(d)} - \left( 1 - \frac{p}{(p^2-1)} \right) \right|, \left| \frac{\# \mathcal{P}_1(d)}{\# \mathcal{P}(d)} - \frac{1}{p} \right|, \left| \frac{\# \mathcal{P}_2(d)}{\# \mathcal{P}(d)} - \frac{1}{(p^3-p)} \right| \right\} < C_{E[p]} \cdot q^{-\frac{d}{2}},
\end{equation}
where $C_{E[p]} := 6(p^3 + g_{E[p]}) > 0$.
\end{remark}

\subsection{Splitting partition of polynomials over finite fields} \label{section:splitting}

In this subsection, we define the splitting partition with respect to a tuple of integers $(n,w)$, which will help us organize conditions that we wish to impose on irreducible factors of $f \in F_n(\F_q)$.

\begin{definition}
    Let $m < n$ be two positive integers.
    We denote by
    \begin{equation}
        \lambda_{[m,n]} := \left\{ (\lambda_{i,j,k}, i, j, k) \right\}_{m \leq i \leq n, 1 \leq j \leq n, 0 \leq k \leq 2}
    \end{equation}
    a set of $3n(n-m+1)$ many 4-tuples such that all coordinates $\lambda_{i,j,k}, i, j, k$ are non-negative integers satisfying the constraints $\lambda_{i,j,k} \geq 0$, $m \leq i \leq n$, $1 \leq j \leq n$, and $0 \leq k \leq 2$. We also use the abbreviation $\lambda_n := \lambda_{[1,n]}$.
\end{definition}

\begin{definition}
    Throughout the rest of the manuscript, we denote by $\mathfrak{n}$ the positive integer
    \begin{equation}
        \mathfrak{n} := \lfloor \frac{4(m_{n,q})^2}{\log q} \rfloor = \lfloor \frac{4(\log n + \log \log q)^2}{\log q} \rfloor.
    \end{equation}
\end{definition}

\begin{definition}
    Fix two positive integers $n$ and $w$. 
    We say that $\lambda_n$ is a splitting partition with respect to $(n,w)$ if it satisfies the following two conditions.
    \begin{enumerate}
        \item $\sum_{i=1}^n \sum_{j=1}^n \sum_{k=0}^2 \lambda_{i,j,k} \cdot i \cdot j = n$.
        \item $\sum_{i=1}^n \sum_{j=1}^n \sum_{k=0}^2 \lambda_{i,j,k} = w$.
    \end{enumerate}
    We say that a polynomial $f$ over $\mathbb{F}_q$ admits a splitting partition $\lambda_n$ with respect to $(n,w)$ if the following three conditions are satisfied.
    \begin{enumerate}
        \item The degree of $f$ is equal to $n$.
        \item The number of distinct irreducible factors of $f$ is equal to $w$.
        \item For all integers $1 \leq i \leq n$, $1 \leq j \leq n$, and $0 \leq k \leq 2$, there are $\lambda_{i,j,k}$ many distinct irreducible polynomials $g_1, g_2, \cdots, g_{\lambda_{i,j,k}}$ of degree $i$ in $\mathcal{P}_k$ such that $g^j \mid f$ but $g^{j+1} \nmid f$.
    \end{enumerate}
    More concretely, if $f$ admits an irreducible factorization
    \begin{equation*}
        f = g_1^{j_1} g_2^{j_2} \cdots g_w^{j_w},
    \end{equation*}
    such that each irreducible factor $g_\ell$ is an element of $\mathcal{P}_{k_\ell}(i_\ell)$, then a splitting partition $\lambda_n$ with respect to $(n,w)$ is determined from
    \begin{equation*}
        \lambda_{i,j,k} := \# \left\{ g_\ell \text{ irreducible} : \deg g_\ell = i, \; g_\ell^j \mid f, \; g_\ell^{j+1} \nmid f, \; g_\ell \in \mathcal{P}_k \right\}.
    \end{equation*}
\end{definition}

For example, if the irreducible factorization of a degree $6$ polynomial $f$ over $\mathbb{F}_q$ is given by $f = g_1^2 g_2 g_3$ such that $g_1 \in \mathcal{P}_1(1)$ and $g_2, g_3 \in \mathcal{P}_2(2)$, then $f$ admits a splitting partition $\lambda_6 := \{(\lambda_{i,j,k},i,j,k)\}$ with respect to $(n,w) = (6,3)$ that satisfies
\begin{equation}
    \lambda_{i,j,k} = \begin{cases}
        2 &\text{ if } i=2, j=1, k=2, \\
        1 &\text{ if } i=1, j=2, k=1, \\
        0 &\text{ otherwise}.
    \end{cases}
\end{equation}

We introduce four properties of splitting partitions with respect to $(n,w)$ which will be of use in subsequent sections.

\begin{definition}
    Let $\lambda_n$ be a splitting partition with respect to $(n,w)$.
    \begin{enumerate}
        \item We say that $\lambda_n$ is $p$-th power free if
        \begin{equation}
            \lambda_{i,j,k} = 0 \text{ whenever } j \geq p.
        \end{equation}
        In other words, any polynomial $f \in F_n(\mathbb{F}_q)$ admitting a $p$-th power free partition $\lambda_n$ is a $p$-th power free polynomial over $\mathbb{F}_q$.
        \item We say that $\lambda_n$ is admissible if it satisfies
        \begin{equation}
            \lambda_{i,j,k} = 0 \text{ whenever } i \leq \mathfrak{n}.
        \end{equation}
        In other words, any polynomial $f \in F_n(\mathbb{F}_q)$ admitting an admissible partition $\lambda_n$ is not divisible by irreducible polynomials of degree at most $\mathfrak{n}$.
        \item We say that $\lambda_n$ is forgettable if
        \begin{equation}
            \lambda_{i,j,k} = 0 \text{ whenever } i > \mathfrak{n}.
        \end{equation}
        In other words, any polynomial $f \in F_n(\mathbb{F}_q)$ admitting a forgettable partition $\lambda_n$ is not divisible by irreducible polynomials of degree greater than $\mathfrak{n}$.
        \item We say that an admissible partition $\lambda_n$ is locally arrangeable if
        \begin{equation}
            \lambda_{i,j,0} \neq 0 \text{ for some } i > N \text{ and } j \not\equiv 0 \text{ mod } p.
        \end{equation}
        Any polynomial $f \in F_n(\mathbb{F}_q)$ admitting a locally arrangeable partition has an irreducible factor in $\mathcal{P}_0$ of degree greater than $\mathfrak{n}$ and of multiplicity coprime to $p$.
    \end{enumerate}
\end{definition}

\begin{definition}
    We define the following set of splitting partitions with respect to a tuple of positive integers $(n,w)$.
    \begin{itemize}
        \item $\Lambda_{n,w} := \{\lambda_n \; | \; \lambda_n \text{ is a splitting partition with respect to } (n,w)\}$.
        \item $\Lambda_{n,w}^{ad} := \{\lambda_n \in \Lambda_{n,w} \; | \; \lambda_n \text{ is a p-th power free admissible partition} \}$.
        \item $\Lambda_{n,w}^{for} := \{\lambda_n \in \Lambda_{n,w} \; | \; \lambda_n \text{ is a forgettable partition}\}$.
        \item $\Lambda_{n,w}^{la} := \{\lambda_n \in \Lambda_{n,w}^{ad} \; | \; \lambda_n \text{ is a locally arrangeable partition}\}$.
    \end{itemize}
\end{definition}

Using these splitting partitions, we further decompose the set $F_n(\mathbb{F}_q)$ of monic polynomials of degree $n$ as follows.

\begin{definition}
    Given a polynomial $f \in F_n(\F_q)$ and an irreducible polynomial $g$ over $\F_q$, denote by $v_g(f)$ the multiplicity of $g$ as an irreducible factor of $f$. We define
    \begin{align}
        \begin{split}
            f^* := \prod_{\substack{g \mid f \\ g \in \cup_{i=\mathfrak{n}+1}^n \mathcal{P}(d)}} g^{v_g(f)}, \; \; \; \; f_* &:= \prod_{\substack{g \mid f \\ g \in \cup_{i=1}^\mathfrak{n} \mathcal{P}(d)}} g^{v_g(f)}.
        \end{split}
    \end{align}
    We note that $f = f^* f_*$, where the irreducible factors of $f^*$ are all of degree greater than $\mathfrak{n}$ (and likewise for $f_*$).
\end{definition}

\begin{definition} \label{defn:polynomial_sets}
    Let $n,w$ be two positive integers. Given a polynomial $f \in F_n(\F_q)$, denote by $w(f)$ the number of distinct irreducible factors of $f$.
    \begin{enumerate}
        \item Given a positive integer $w' < w$, we denote by
        \begin{equation}
            F_{n,(w,w')}(\F_q) := \{f \in F_n(\F_q) \; | \; w(f) = w \text{ and } w(f^*) = w' \}.
        \end{equation}
        \item Given a positive integer $N < n$, we denote by
        \begin{equation}
            F_{(n,N), (w,w')}(\F_q) := \{f \in F_{n,(w,w')}(\F_q) \; | \; \deg f^* = N \text{ and } f^* \text{ is } p\text{-th power free}\}.
        \end{equation}
        \item Given a locally arrangeable partition $\lambda \in \Lambda_{N,w'}^{la}$ and a forgettable partition $\eta \in \Lambda_{n-N,w-w'}^{for}$, we denote by
        \begin{equation}
            F_{(n,N),(w,w')}^{(\lambda,\eta)}(\F_q) := \{f \in F_{(n,N),(w,w')}(\F_q) \; | \; f^* \text{ admits } \lambda, \; f_* \text{ admits } \eta \}.
        \end{equation}
        \item We denote by $\hat{F}_{(n,N),(w,w')}(\F_q)$ the following subset of $F_{(n,N),(w,w')}(\F_q)$:
        \begin{equation}
        \hat{F}_{(n,N),(w,w')}(\F_q) := \bigsqcup_{\lambda \in \Lambda_{N,w'}^{la}} \bigsqcup_{\eta \in \Lambda_{n-N,w-w'}^{for}} F_{(n,N),(w,w')}^{(\lambda,\eta)}(\F_q).
    \end{equation}
    \end{enumerate}
\end{definition}

\begin{remark}
    The construction of $F_{(n,N),(w,w')}^{(\lambda,\eta)}(\F_q)$ is closely related to the construction of fan structure from \cite[Chapter 2,3,4]{KMR14}. Given two sets $B$ and $C$, denote by 
    \begin{equation}
        B * C := \{\{\delta\} \cup \{q\} \; | \; \delta \in B, q \in C \setminus \{q\}\},
    \end{equation}
    as stated in \cite[Chapter 4, Page 1085]{KMR14}. Note that if $B \cap C = \emptyset$, then $B * C = B \times C$.
    For any positive integer $m > 0$, inductively define
    \begin{align}
        \begin{split}
            \mathcal{P}_k(i)^{*1} &= \mathcal{P}_k(i), \\
            \mathcal{P}_k(i)^{*m} &= \mathcal{P}_k(i)^{*(m-1)} * \mathcal{P}_k(i).
        \end{split}
    \end{align}
    Then one has
    \begin{equation}
        F_{(n,N),(w,w')}^{(\lambda,\eta)}(\F_q) = \left[\prod_{i,j,k} \mathcal{P}_k(i)^{*\lambda_{i,j,k}}\right] \times \left[\prod_{\hat{i},\hat{j},\hat{k}} \mathcal{P}_k(\hat{i})^{*\eta_{\hat{i},\hat{j},\hat{k}}} \right].
    \end{equation}
\end{remark}

To understand how the sizes of four types of subsets of $F_n(\F_q)$ are related to each other, we prove the following proposition, which shows that for sufficiently large $n$, any monic polynomial of degree $d$ cannot have too many factors whose degree is at most $\mathfrak{n}$.

\begin{proposition} \label{proposition:f_mu_irred}
    Suppose $m_{n,q} := \log n + \log \log q$ satisfies the condition that $m_{n,q} > e^{e^e}$. Let $\epsilon = \frac{1}{\log \log m _{n,q}}$. Then
    \begin{equation}
        \# \{f \in F_n(\F_q) \; | \; w(f_*) > \epsilon m_{n,q}\} < 4 \cdot q^n \cdot (n \log q)^{-(\log m_{n,q})^{1 - \sqrt{\epsilon}}}.
    \end{equation}
\end{proposition}
\begin{proof}
We thank the reviewer for suggesting the following strategy of the proof. Let $\mathcal{Q}$ be a set of irreducible monic polynomials of degree at most $n$. Using the fact that the number of monic polynomials of degree $n$ over $\mathbb{F}_q$ that is divisible by an irreducible polynomial $g$ is at most $q^{n - \deg(g)}$, we can deduce that the number of monic polynomials of degree $n$ with at least $r$ distinct irreducible factors from $\mathcal{Q}$ is at most 
\begin{equation} \label{prop5.4:eq1}
    q^n \cdot \frac{1}{r!} \cdot \left( \sum_{g \in \mathcal{Q}} q^{-\deg(g)} \right).
\end{equation}

For our purposes, we let
\begin{equation}
    \mathcal{Q} := \cup_{i=1}^{\mathfrak{n}} \mathcal{P}(i),
\end{equation}
where we recall that $m_{n,q} := \log n + \log \log q$ and $\mathfrak{n} := \lfloor \frac{4(\log n + \log \log q)^2}{\log q} \rfloor = \lfloor \frac{4 m_{n,q}^2}{\log q} \rfloor$. Then the prime number theorem for global function fields implies
\begin{equation}
    \sum_{g \in \mathcal{Q}} q^{-\deg g} = \sum_{i=1}^{\mathfrak{n}} \# \mathcal{P}(i) \cdot q^{-i} \leq 2 \cdot \sum_{i=1}^{\mathfrak{n}} \frac{1}{i} \leq 2 \log(\mathfrak{n}) + 2 \leq 4 \log m_{n,q} + 4 \log 2 + 2.
\end{equation}

Suppose that $m_{n,q}> e^{e^e}$. We let
\begin{equation}
    r := \epsilon m_{n,q}, \; \;  \epsilon := \frac{1}{\log \log m_{n,q}}.
\end{equation}
Stirling's approximation theorem shows that for such $n$ satisfying $m_{n,q} > e^{e^e}$,
\begin{align}
    \begin{split}
        \frac{1}{r!} &< \frac{1}{\sqrt{2 \pi r} \left( \frac{r}{e} \right)^r} \\
        &= \frac{1}{\sqrt{2\pi \epsilon m_{n,q}}} \cdot (n \log q)^{-\epsilon \log m_{n,q} - \epsilon \log \epsilon + \epsilon}.
    \end{split}
\end{align}
We note that because $0 < \epsilon < 1$, it follows that $0 < \epsilon - \epsilon \log \epsilon < 1$. Hence, the above equation can be simplified as
\begin{align}
    \begin{split}
        \frac{1}{r!} &< \frac{1}{\sqrt{\pi m_{n,q}}} \cdot (n \log q)^{-\epsilon \log m_{n,q} + 1}.
    \end{split}
\end{align}
Combining with equation (\ref{prop5.4:eq1}), we obtain
\begin{align}
\begin{split}
    \# \{f \in F_n(\F_q) \; | \; w(f_*) > \epsilon m_{n,q}\} &< q^n \cdot \frac{4 \log m_{n,q} + 4 \log 2 + 2}{\sqrt{\pi m_{n,q}}} \cdot (n \log q)^{-\epsilon \log m_{n,q} + 1} \\
    &< q^n \cdot 4 \cdot (n \log q)^{-\epsilon \log m_{n,q} + 1}.
\end{split}
\end{align}
The statement of the proposition follows from the inequality that whenever $m_{n,q} > e^{e^e}$, we have $\epsilon \log m_{n,q} - 1 > (\log m_{n,q})^{1 - \sqrt{\epsilon}}$.
\end{proof}

We now show that the set $F_n(\F_q)$ can be approximated by disjoint union of subsets of form $F_{(n,N),(w,w')}^{(\lambda,\eta)}(\F_q)$ where $\lambda$ is a locally arrangeable splitting partition, and $\eta$ is a forgettable splitting partition.

\begin{proposition} \label{proposition:fan_approximation}
    Let $\rho \in (0,1)$ be a positive number. Suppose $n$ is a positive integer such that $m_{n,q} > \max \{e^{e^e}, \log6 + \log(p^3 + g_{E[p]})\}$. Let $\epsilon = \frac{1}{\log \log m _{n,q}}$. Then
    \begin{align}
    \begin{split}
        & \; \; \; \; \# F_n(\F_q) - \sum_{w = \rho m_{n,q}}^{2 m_{n,q}} \sum_{w' = (1-\epsilon)w}^w \sum_{N = w'\mathfrak{n}}^n \# \hat{F}_{(n,N),(w,w')}(\F_q) \\
        &\leq 4 \cdot q^n \cdot \max \left(n^{-\rho \log \rho - 1 + \rho}, 3 m_{n,q}^2 \cdot \left( \frac{p}{p^2-1} \right)^{(1-\epsilon)\rho m_{n,q}}  \right).
    \end{split}
    \end{align}
\end{proposition}
In other words, the above proposition shows that given $\rho \in (0,1)$, almost every monic polynomial $f$ of degree $n$ satisfies:
\begin{enumerate}
    \item The number of distinct irreducible factors of $f$ is between $\rho m_{n,q}$ and $2m_{n,q}$.
    \item The number of distinct irreducible factors of $f$ of degree at most $\mathfrak{n}$ is at most $(1-\epsilon)\rho m_{n,q}$ for some small enough $\epsilon > 0$
    \item The polynomial $f^*$ is $p$-th power free, and has at least $1$ irreducible factor inside $\mathcal{P}_0$ of degree at least $\mathfrak{n}$.
\end{enumerate}
The two error terms appearing in Proposition \ref{proposition:fan_approximation} correspond to two of the error terms constituting the constant $\alpha(p)$ defined in Theorem \ref{theorem:main_theorem}.
\begin{proof}
    By Theorem \ref{theorem:effective_erdos_kac} and Proposition \ref{proposition:f_mu_irred}, for any small enough $\epsilon > 0$,
    \begin{align} \label{eq:prop5.9eq1}
        \# F_n(\F_q) - \sum_{w = \rho m_{n,q}}^{2m_{n,q}} \sum_{w' = (1-\epsilon)w}^w \# F_{n,(w,w')}(\F_q) \leq 4 \cdot q^n \cdot n^{-\rho \log \rho - 1 + \rho}.
    \end{align}
    Using the definition of $f^*$, it follows that if $f^*$ is not $p$-th power free, then the degree of the $p$-th power free part of $f^*$ is at most $n - p \mathfrak{n}$. Therefore, one obtains that
    \begin{align} \label{eq:prop5.9eq2}
        \# F_{n,(w,w')}(\F_q) - \sum_{N = w' \mathfrak{n}}^n \# F_{(n,N),(w,w')}(\F_q) \leq q^n \cdot n^{-4(p-1)(\log n)^2}.
    \end{align}
    Using the definition of $\Lambda_{n,w}$ it follows that for any four integers $n > N$ and $w > w'$,
    \begin{equation}
        F_{(n,N),(w,w')}(\F_q) = \bigsqcup_{\lambda \in \Lambda_{N,w'}^{ad}} \bigsqcup_{\eta \in \Lambda_{n-N,w-w'}^{for}} F_{(n,N),(w,w')}^{(\lambda,\eta)}(\F_q).
    \end{equation}
    Recall that $g_{E[p]}$ is the genus of the global function field $K(E[p])/K$. Because we assumed that $m_{n,q} > \max\{e^{e^e}, \log6 + \log(p^3 + g_{E[p]})\}$, we obtain that
    \begin{equation*}
        q^{\frac{\mathfrak{n}+1}{4}} > (n \log q)^{m_{n,q}} > e^{m_{n,q}} > 6(p^3 + g_{E[p]}).
    \end{equation*}
    Suppose that $w' \leq 2m_{n,q}$. Apply Theorem \ref{theorem:effective_chebotarev} with respect to the field $K(E[p])/K$ to get
    \begin{align} \label{eq:prop5.9eq3}
    \begin{split}
        & \; \; \; \; \sum_{N = w' \mathfrak{n}}^n \left( \# F_{(n,N),(w,w')}(\F_q) - \sum_{\lambda \in \Lambda_{N,w}^{la}} \sum_{\eta \in \Lambda_{n-N,w-w'}^{for}} \# F_{(n,N),(w,w')}^{(\lambda,\eta)}(\F_q) \right) \\
        &\leq q^n \cdot \left( \left( \frac{p}{p^2-1} \right)^{w'} + \sum_{k=1}^\infty (n \log q)^{(-m_{n,q}+2)k} \right) \\
        &\leq q^n \cdot \left( \left( \frac{p}{p^2-1} \right)^{w'} + 2 \cdot (n \log q)^{-m_{n,q}+2} \right) \leq 3 \cdot q^n \cdot \left(\frac{p}{p^2-1}\right)^{w'}.
    \end{split}
    \end{align}
    The quantity $\left(\frac{p}{p^2-1}\right)^{w'}$ is the leading term of the probability that none of the irreducible factors of $f^*$ are in $\mathcal{P}_0$, and the rest of the terms are obtained from the rate of convergence of the Chebotarev density theorem and binomial theorem, in particular equation (\ref{equation:prime_split_density_general}).  Combining equations (\ref{eq:prop5.9eq1}), (\ref{eq:prop5.9eq2}), and (\ref{eq:prop5.9eq3}), we obtain the statement of the proposition.
\end{proof}

\subsection{Equidistribution of local characters} \label{section:equidistribution}
In this subsection, we prove that for sufficiently large $n$, the probability distribution that the set of global cyclic order-$p$ characters induced from the set of irreducible polynomials of degree $n$ restricts to a uniform distribution over the set of finite Cartesian products of local unramified cyclic order-$p$ characters at finitely many places of degree strictly less than $n$.

\begin{theorem}{\cite[Theorem 2.1]{Hsu98}} \label{theorem:character_sum}
Let $h$ be any square-free polynomial over $\F_q$. Let $\chi_h$ be a non-trivial character $\chi: (\F_q[t]/h)^\times \to \mathbb{C}^\times$. Then
\begin{equation}
    \sum_{v \in \mathcal{P}(i)} \chi(v) \leq (\deg h + 1) \frac{q^{\frac{i}{2}}}{i}.
\end{equation}
\end{theorem}

An immediate corollary of the theorem above is that the effective error bounds of the density of whether the restriction of a global cyclic order-$p$ character associated to an irreducible polynomial forms a uniform distribution over the set of finite cartesian products of local unramified cyclic characters is given by the order of $q^{-\frac{n}{2}}$.

\begin{corollary} \label{corollary:quadratic_character_sum}
Let $K = \F_q(t)$ be a global function field such that $\mu_p \subset \F_q$. Let $h_1, h_2, \cdots, h_w$ be irreducible polynomials over $\F_q$. Given a place $v$ of degree $i$, denote by $\legendre{v}{h_k}_{p} \in \mu_p$ the $p$-th power residue symbol. Then for any $a \in \mu_p^{\oplus w}$,
\begin{equation}
    \left| \frac{\# \{v \in \mathcal{P}(i) \; | \; \left(\legendre{v}{h_k}_p \right)_{k=1}^w = a \in \mu_p^{\oplus w} \} }{\# \mathcal{P}(i)} - \frac{1}{p^w} \right| < \left( \sum_{k=1}^w \deg h_k + 1 \right) \cdot q^{-i/2}/i.
\end{equation}
\end{corollary}
\begin{proof}
We thank the reviewer for suggesting the strategy of the proof outlined as follows. 

For any abelian group $H$ and $\Omega:= \{\chi: H \to \mathbb{C}\}$ the set of characters of $H$, the orthogonality of characters imply that
\begin{equation}
    \sum_{\chi \in \Omega} \frac{\chi(g_1)}{\chi(g_2)} = \begin{cases}
        |H| & \text{ if } g_1 = g_2 \\
        0 & \text{ otherwise }.
    \end{cases}
\end{equation}
We let $H$ to be the abelian group isomorphic to $\mu_p^{\oplus w}$ generated by the Legendre symbols
\begin{equation}
    \left\{ \legendre{\cdot}{h_1}_p, \legendre{\cdot}{h_2}_p, \cdots, \legendre{\cdot}{h_w}_p \right\}.
\end{equation}
Suppose $g_2 = a \in \mu_p^{\oplus w}$. Using the orthogonality of characters, we obtain
\begin{align}
\begin{split}
    & \; \; \; \; \sum_{v \in \mathcal{P}(i)} \sum_{\chi \in \Omega} \frac{\chi \left( \legendre{v}{h_1}_p, \legendre{v}{h_2}_p, \cdots, \legendre{v}{h_w}_p \right)}{\chi(a)} = \# \left\{v \in \mathcal{P}(i) \; | \; \left(\legendre{v}{h_k}_p \right)_{k=1}^w = a \right\} \cdot p^w.
\end{split}
\end{align}
The left hand side of the above equation can be rewritten as
\begin{equation}
    = \# \mathcal{P}(i) + \sum_{\substack{\chi \in \Omega \\ \chi \neq id}} \sum_{v \in \mathcal{P}(i)} \frac{\chi \left( \legendre{v}{h_1}_p, \legendre{v}{h_2}_p, \cdots, \legendre{v}{h_w}_p \right)}{\chi(a)}.
\end{equation}
Using Theorem \ref{theorem:character_sum}, the summands of the second terms have absolute values bounded above by $\left( \sum_{k=1}^w \deg(h_k) + 1 \right) \cdot q^{i/2}/i$. Hence, we obtain that
\begin{equation}
    \left| \frac{\# \{v \in \mathcal{P}(i) \; | \; \left(\legendre{v}{h_k}_p \right)_{k=1}^w = a \in \mu_p^{\oplus w} \} }{\# \mathcal{P}(i)} - \frac{1}{p^w} \right| < (\sum_{k=1}^w \deg(h_i) + 1) \cdot \frac{q^{-i/2}}{i}.
\end{equation}
\end{proof}

We also prove that given a choice of an elliptic curve $E/K$, the equidistribution of characters still holds for subsets of places $v$ inside $\mathcal{P}_0(i)$, $\mathcal{P}_1(i)$, and $\mathcal{P}_2(i)$.
\begin{corollary} \label{corollary:P0_quadratic_character_sum}
    Let $E$ be an elliptic curve over $K$ satisfying conditions in (\ref{equation:assumption_local_twists}). Suppose that $h_1, h_2, \cdots, h_w$ are irreducible polynomials over $\F_q$. Let $n$ be an integer such that $\sum_{\ell=1}^w \deg h_\ell \leq n$ and $w \leq 2 m_{n,q}$. 
    \begin{enumerate}
        \item Suppose $p \geq 5$, or $K(\sqrt[p]{h_1}, \cdots, \sqrt[p]{h_w}) \cap K(E[p]) = K$. Then for any element $a \in \mu_p^{\oplus w}$, and $i > \mathfrak{n}$, there exists a constant $\hat{C}_{E,p,q} > 0$ depending only on $E$, $p$, $q$ such that 
    \begin{equation}
        \left| \frac{\# \{ v \in \mathcal{P}_k(i) \; | \; \left(\legendre{v}{h_\ell}_p \right)_{\ell=1}^w = a \in \mu_p^{\oplus w}\}}{\# \mathcal{P}_k(i)} - \frac{1}{p^w} \right| < \hat{C}_{E,p,q} \cdot (n \log q)^{-2 m_{n,q} + 2\log p}.
    \end{equation}
        \item Suppose $p = 2, 3$ and $K(\sqrt[p]{h_1}, \cdots, \sqrt[p]{h_w}) \cap K(E[p]) \neq K$. Then for any $i > \mathfrak{n}$, there are $p^w - p^{w-1}$ many elements $a \in \mu_p^{\oplus w}$ such that $\left(\legendre{v}{h_\ell}_p \right)_{\ell=1}^w \neq a$ for all $v \in \mathcal{P}_k(i)$. For the other $p^{w-1}$ many elements $a \in \mu_p^{\oplus w}$, there exists a constant $\hat{C}_{E,p,q} > 0$ depending only on $E$, $p$, $q$ such that 
    \begin{equation}
        \left| \frac{\# \{ v \in \mathcal{P}_k(i) \; | \; \left(\legendre{v}{h_\ell}_p \right)_{\ell=1}^w = a \in \mu_p^{\oplus w}\}}{\# \mathcal{P}_k(i)} - \frac{1}{p^{w-1}} \right| < \hat{C}_{E,p,q} \cdot (n \log q)^{-2 m_{n,q} + 2\log p}.
    \end{equation}
    \end{enumerate}
\end{corollary}
\begin{proof}
    Given an irreducible polynomial $h$ over $\F_q$, consider the cyclic order-$p$ abelian extension $K(\sqrt[p]{h})/K$. Then if $v$ is coprime to $h$, then the $p$-th power residue symbol $\legendre{v}{h}_p$ defines the action of the Frobenius element $\Frob_v$ on $\sqrt[p]{h}$ via 
\begin{equation*}
    \Frob_v(\sqrt[p]{h}) = \legendre{v}{h}_p \sqrt[p]{h},
\end{equation*}
which in fact originates from the definition of the Artin reciprocity map, see \cite[Chapter 3, Chapter 10]{Ro02} for a detailed description.

With the irreducible polynomials $h_1, h_2, \cdots, h_w$ as stated, consider the field extension $L := K(E[p], \sqrt[p]{h_1}, \cdots, \sqrt[p]{h_w})$. Suppose that $K(\sqrt[p]{h_1}, \cdots, \sqrt[p]{h_w}) \cap K(E[p]) = K$. Note that this condition always holds for any choice of irreducible polynomials $h_i$ if $p \geq 5$, because $\text{SL}_2(\mathbb{F}_p)$ has no normal subgroup of index $p$. It hence follows that
\begin{equation}
    \text{Gal}(L/K) \cong \SL_2(\F_p) \times \mu_p^{\oplus w}
\end{equation}
and its conjugacy classes are of form $C \times \{a\}$, where $C \subset \SL_2(\F_p)$ is a conjugacy class and $a \in \mu_p^{\oplus w}$ is an element. Recall that
\begin{equation}
    \# \text{Gal}(L/K) = p^w \cdot (p^3 - p).
\end{equation}
By Riemann-Hurwitz theorem,
\begin{equation}
    g_L \leq p^w \cdot (2g_{E[p]} - 2 + p^3),
\end{equation}
where $g_{E[p]}$ is the genus of the global function field $K(E[p])$. Applying Corollary \ref{corollary:effective_chebotarev} and Corollary \ref{corollary:quadratic_character_sum} proves the first statement of the theorem, where we set $\hat{C}_{E,p,q} := 6(2g_{E[p]} + 2p^3 - p - 2)$.

The case where $K(\sqrt[p]{h_1}, \cdots, \sqrt[p]{h_w}) \cap K(E[p]) \neq K$ occurs when $p = 2$ or $3$. In such cases, the field extension $K(\sqrt[p]{h_1}, \cdots, \sqrt[p]{h_w}) \cap K(E[p])$ is a non-trivial cyclic Galois extension over $K$ of degree $p$, which corresponds to the normal subgroup of $\text{SL}_2(\mathbb{F}_p)$ of index $p$. It then follows that
\begin{equation}
    \text{Gal}(L/K) \cong \text{SL}_2(\mathbb{F}_p) \times \mu_p^{\oplus w-1}.
\end{equation}
Applying the analogous argument for proving the first statement of the theorem yields the rest of the results.
\end{proof}

\begin{remark}
    Suppose that $p = 2$. The criterion to determine which elements $a \in \mu_p^{\oplus w}$ satisfy $\left(\legendre{v}{h_\ell}_p \right)_{\ell=1}^w \neq a$ for all $v \in \mathcal{P}_k(i)$ can be determined by what is called the ``sign function'', see \cite[Definition 10.6]{KMR14} for further details.
\end{remark}

\section{Local Selmer groups} \label{section:quadratic_twists}

The objective of this section focuses on defining what is called the local Selmer groups of $E$ associated to a cyclic order $p$ local character, and understanding their dimensions over the subset of polynomials $F_{(n,N),(w,w')}^{(\lambda,\eta)}(\F_q)$. These results will be of relevant use in Section \ref{section:global_selmer}, where we will understand the dimensions of $\Sel_\pi(E^{\chi_f})$ as $f$ ranges over $F_n(\F_q)$.

\subsection{Local twists} \label{section:local_twists}
The constructions and properties of the local Selmer groups, as explored in \cite{MR07, KMR13, KMR14}, rests upon utilizing results regarding Galois cohomology groups and Poitou-Tate duality theorems over number fields, the theories of which also hold valid over global function fields $\F_q(t)$, see for example Chapter 1 of \cite{Mi06} for a rigorous treatment of Poitou-Tate duality theorems for global function fields. We further enrich these results by using the properties that hold over $\F_q(t)$ explored from Section \ref{section:effective_theorems_RH} which are not necessarily proven for number fields.

\begin{definition} \label{definition:basic_local_def}
We introduce the following notations regarding cyclic order $p$ characters $\chi \in \text{Hom}(\text{Gal}(\overline{K_v}/K_v,\mu_p)$, some of which are as stated in \cite[Sections 5, 7, 9]{KMR14}. We recall the sets of primes $\Sigma$ and $\Sigma_E$ associated to choices of $E$ from Definition \ref{defn:sets_places}. Given a set $\Sigma$, we let $\sigma$ be a square-free product of places coprime to elements in $\Sigma$.
\begin{itemize}
    \item $\Omega_\sigma$: the set of finite Cartesian products of local characters
    \begin{equation*}
        \chi := (\chi_v)_v \in \Omega_\sigma := \prod_{v \in \Sigma(\sigma)} \text{Hom}(\Gal(\overline{K}_v/K_v), \mu_p)
    \end{equation*}
    such that the component $\chi_v$ is ramified if $v \mid \sigma$. For the sake of convenience, we will denote by $\text{Hom}_{unr}(\Gal(\overline{K}_v/K_v), \mu_p)$ the set of unramified local characters at place $v$, and by $\text{Hom}_{ram}(\Gal(\overline{K}_v/K_v), \mu_p)$ the set of ramified local characters at place $v$. Assuming that $\mu_p \subset K_v$, there are $p$ distinct unramified local characters at $v$, and $p(p-1)$ distinct ramified local characters at $v$.
    \item $\Omega_E$: the set of finite Cartesian products of local characters
    \begin{equation*}
        \chi := (\chi_v)_v \in \Omega_E := \prod_{\substack{v \in \Sigma_E }} \text{Hom}(\Gal(\overline{K}_v/K_v), \mu_p).
    \end{equation*}

    \item Fix an element $\chi \in \Omega_\sigma$. Let $\mathfrak{v}$ be a place over $K$ such that $\mathfrak{v} \not\in \Sigma(\sigma)$. Let $\chi' \in \Omega_{\sigma \mathfrak{v}}$ be an element such that
    \begin{itemize}
    \item For any $v \in \Sigma(\sigma)$, $\chi'_v = \chi_v$.
    \item At $\mathfrak{v}$, $\chi'_\mathfrak{v}$ is ramified.
    \end{itemize} 
    Denote by $\Omega_{\chi,\mathfrak{v}}$ the set of local characters $\chi'$ satisfying the two conditions above. Note that
    \begin{equation*}
    \Omega_{\sigma \mathfrak{v}} = \bigsqcup_{\chi \in \Omega_\sigma} \Omega_{\chi,\mathfrak{v}}.
    \end{equation*}
\end{itemize}
\end{definition}

\begin{definition} \label{definition:consecutive_local_char}

We introduce the following notations regarding local Selmer groups of $E$ associated to cyclic order $p$ characters $\chi \in \text{Hom}(\text{Gal}(\overline{K_v}/K_v,\mu_p)$, some of which are as stated in \cite[Sections 5, 7, 9]{KMR14}.

\begin{itemize}
\item Given a Cartesian product of local characters $\chi \in \Omega_\sigma$, the local Selmer group of $E$ associated to $\chi$ is denoted as
    \begin{equation}
        \Sel(E[p], \chi) := \text{Ker} \left( H^1_{\et}(K, E[p]) \to \prod_v H^1_{\et}(K_v, E[p])/\mathcal{H}^{\chi}_v \right),
    \end{equation}
    where
    \begin{equation} \label{eqn:defn-sbsp}
        \mathcal{H}^{\chi}_v := \begin{cases}
            \text{im} \left(\delta_v^\chi: E^{\chi_v}(K_v)/\pi E^{\chi_v}(K_v) \to H^1(K_v, E[p]) \right) &\text{ if } v \in \Sigma(\sigma)\\
            H^1(\mathcal{O}_{K_v}, E[p]) &\text{ if } v \not\in \Sigma(\sigma).
        \end{cases}
    \end{equation}
    Under all but the third assumption stated in (\ref{equation:assumption_local_twists}), we use the isomorphism
    \begin{align*}
        H^1_{\et}(K, E[p]) &\cong H^1_{\et}(K, E^\chi[\pi]), \\
        H^1_{\et}(K_v, E[p]) &\cong H^1_{\et}(K_v, E^\chi_v[\pi]),
    \end{align*}
    to define the local Selmer group $\Sel(E[p], \chi)$, see in particular \cite[Proposition 4.1, Definition 4.3]{MR07}. Even though the reference particularly constructs these groups over number fields, the relevant results extend to global function fields as well.
    \item We recall that the Weil pairing $E[p] \times E[p] \to \mu_p$ and the cup product on $H^1_{\et}(K_v, E[p])$ induce a symmetric pairing
    \begin{equation*}
        H^1_{\et}(K_v, E[p]) \times H^1_{\et}(K_v,E[p]) \to \mathbb{F}_p.
    \end{equation*}
    Denote by $q_v$ the quadratic form induced from the symmetric pairing stated above. Then $\mathcal{H}_v^\chi$ is a maximal isotropic subspace of $H^1_{\et}(K_v,E[p])$ with respect to $q_v$. Furthermore, if $v \in \Sigma(\sigma) \setminus \Sigma$, then $\mathcal{H}_v^\chi \cap H^1(\mathcal{O}_{K_v}, E[p]) = 0$. We refer to \cite[Section 4.2]{PR12} and \cite[Section 5, Proposition 6.4]{KMR14} for details of the proof of these observations.
    \item If $v \in \mathcal{P}_0$, then $\mathcal{H}_v^{\chi}$ is trivial. If $v \in \mathcal{P}_1 \cap \Sigma(\sigma)$, then there is a unique $1$-dimensional ramified subspace, denoted as $\mathcal{H}^1_{ram}$. If $v \in \mathcal{P}_2 \cap \Sigma(\sigma)$, then there are $p$ distinct $2$-dimensional ramified subspaces $\mathcal{H}_v^{\chi}$, each corresponding to a tamely totally ramified cyclic $p$ extension $\overline{K}_v^{\text{Ker}(\chi_v)}$ over $K_v$. As stated in \cite[Definition 5.10]{KMR14}, for such a $v$ we have a set-theoretic bijection 
    \begin{equation*}
        \alpha_v: \frac{\text{Hom}_{ram}(\text{Gal}(\overline{K_v}/K_v),\mu_p)}{\text{Aut}(\mu_p)} \to \{\mathcal{H}_v^\chi\}_{\chi \in \text{Hom}_{ram}(\text{Gal}(\overline{K_v}/K_v),\mu_p)}. 
    \end{equation*}
    These identifications allow us to rewrite the subspaces $\mathcal{H}_v^\chi$ appearing in equation (\ref{eqn:defn-sbsp}) as
    \begin{equation} \label{eqn:defn-sbsp2}
        \mathcal{H}_v^\chi := \begin{cases}
            \alpha_v(\overline{K}_v^{\text{Ker}(\chi_v)}) &\text{ if } v \in \mathcal{P}_2 \cap \Sigma(\sigma), \\
            \mathcal{H}^1_{ram} &\text{ if } v \in \mathcal{P}_1 \cap \Sigma(\sigma), \\
            0 &\text{ if } v \in \mathcal{P}_0 \cap \Sigma(\sigma), \\
            \text{im} \delta_v^\chi &\text{ if } v \in \Sigma \setminus \mathcal{P} \text{ and } \mathcal{H}_v = \overline{K}_v^{\text{Ker}(\chi_v)},\\
            H^1(\mathcal{O}_{K_v}, E[p]) &\text{ if } v \not\in \Sigma(\sigma).
            \end{cases}
    \end{equation}
    \item Given a set of local characters $\chi \in \Omega_\sigma$, we denote by $\text{rk}(\chi)$ the dimension of $\Sel(E[p],\chi)$ as an $\F_p$-vector space. By the identification of $\mathcal{H}_v^\chi$ above, we have $\text{rk}(\chi) = \text{rk}(\chi')$ if the following two conditions are satisfied:
    \begin{itemize}
        \item $\text{Ker}(\chi_v) = \text{Ker}(\chi'_v) \subset \text{Gal}(\overline{K}_v/K_v)$ for every $v \in \mathcal{P}_2 \cap \Sigma(\sigma)$.
        \item $\text{rk}(\hat{\chi}) = \text{rk}(\hat{\chi'})$, where $\hat{\chi} := (\chi_v)_{v \in \Sigma_E} \in \Omega_E$ (and likewise for $\hat{\chi'}$).
    \end{itemize}
    Any changes in local conditions over places $v \in \mathcal{P}_0$ do not affect the values of $\text{rk}(\chi)$.
\item Denote by $t_{\chi}(\mathfrak{v})$ the dimension of the image of the local Selmer group $\Sel(E[p],\chi)$ with respect to the localization map at $\mathfrak{v}$, i.e.
\begin{equation}
    t_{\chi}(\mathfrak{v}) := \text{dim}_{\F_p} \text{im} \left(\text{loc}_\mathfrak{v}: \Sel(E[p],\chi) \to H^1(\Oh_{K_\mathfrak{v}},E[p])\right).
\end{equation}
We note that if $\mathfrak{v} \in \mathcal{P}_i$, then $0 \leq t_{\chi}(\mathfrak{v}) \leq i$. Furthermore, we have $t_\chi(\mathfrak{v}) = t_{\chi'}(\mathfrak{v})$ if $\text{Ker}(\chi_v) = \text{Ker}(\chi'_v) \subset \text{Gal}(\overline{K}_v/K_v)$ for every $v \in \Sigma(\sigma)$.
\end{itemize}
\end{definition}

The relation between $t_{\chi}(\mathfrak{v})$ and the differences between ranks of local Selmer groups associated to characters $\chi \in \Omega_\sigma$ and $\chi' \in \Omega_{\chi,\mathfrak{v}}$ is stated in \cite[Proposition 7.2]{KMR14}. 
\begin{proposition} \label{prop:Markov_equi}
    Let $E$ be a non-isotrivial elliptic curve over $K$ satisfying the conditions from equation (\ref{equation:assumption_local_twists}). Fix a square-free product of places $\sigma$ coprime to elements in $\Sigma$, and let $\mathfrak{v}$ be a place of $K$ such that $\mathfrak{v} \not\in \Sigma(\sigma)$. Fix a character $\chi \in \Omega_\sigma$. Then for any $\chi' \in \Omega_{\chi, \mathfrak{v}}$,
    \begin{equation}
        \text{rk}(\chi') - \text{rk}(\chi) = \begin{cases}
            2 & \text{ if } \mathfrak{v} \in \mathcal{P}_2 \text{ and } t_\chi(\mathfrak{v}) = 0 \text{ for exactly } p-1 \text{ many } \chi' \in \Omega_{\chi, \mathfrak{v}},
            \\
            1 & \text{ if } \mathfrak{v} \in \mathcal{P}_1 \text{ and } t_{\chi}(\mathfrak{v}) = 0, \\
            -1 & \text{ if } \mathfrak{v} \in \mathcal{P}_1 \text{ and } t_\chi(\mathfrak{v}) = 1, \\
            -2 & \text{ if } \mathfrak{v} \in \mathcal{P}_2 \text{ and } t_\chi(\mathfrak{v}) = 2, \\
            0 & \text{ otherwise }.
        \end{cases}
    \end{equation}
    We note that the $p-1$ many $\chi' \in \Omega_{\chi,\mathfrak{v}}$ that satisfies $\text{rk}(\chi') - \text{rk}(\chi) = 2$ share an identical cyclic degree $p$ ramified extension over $K_v$.
\end{proposition}
\begin{proof}
    The proof follows from adapting the proof of \cite[Proposition 7.2]{KMR14}. The two conditions required in the statement of \cite[Proposition 7.2]{KMR14}, which are
    \begin{enumerate}
        \item $\text{Pic}(\Oh_{K,\Sigma}) = 0$.
        \item The map $\Oh_{K,\Sigma}^\times / (\Oh_{K,\Sigma}^\times)^p \to \prod_{v \in \Sigma} K_v^\times / (K_v^\times)^p$ is injective.
    \end{enumerate}
    hold regardless of the choice of $\Sigma$ because $\Oh_K = \F_q[t]$ is a Euclidean domain.
\end{proof}

The probability that $t_\chi(\mathfrak{v})$ achieves a certain value can be obtained from a Chebotarev condition over $K$ obtained from $\Sel(E[p],\chi)$, as shown in \cite[Proposition 9.4]{KMR14}.

\begin{proposition}[Local twists of $\pi$-Selmer groups] \label{theorem:local_twists}
Let $E$ be a non-isotrivial elliptic curve over $K$ satisfying the conditions from equation (\ref{equation:assumption_local_twists}). Fix a square-free product of places $\sigma$ coprime to elements in $\Sigma$. Fix a local character $\chi \in \Omega_\sigma$.

Let $d_{i,j}$ be given by the following table:

\begin{center}
    \begin{tabular}{|P{1.5cm}||P{4cm}|P{4cm}|P{5cm}|}
    \hline
    $d_{i,j}$ & $i = 0$ & $i = 1$ & $i = 2$ \\
    \hline
    \hline
    $j = -2$ & $\times$ & $\times$ & $1 - (p+1) p^{-\text{rk}(\chi)} + p^{1-2\text{rk}(\chi)}$ \\
    \hline
    $j = -1$ & $\times$ & $1-p^{-\text{rk}(\chi)}$ & $\times$ \\
    \hline
    $j = 0$ & $1$ & $\times$ & $(p+1)(p^{-\text{rk}(\chi)} -p^{-2\text{rk}(\chi)})$ \\
    \hline
    $j = 1$ & $\times$ & $ p^{-\text{rk}(\chi)}$ & $\times$ \\
    \hline 
    $j = 2$ & $\times$ & $\times$ & $p^{-2\text{rk}(\chi)}$ \\
    \hline
    \end{tabular}
\end{center}

Here, the term "$\times$" denotes the case where such a difference of ranks cannot occur. Let $D_{E,p,q} > 0$ be a constant defined as
\begin{equation}
    D_{E,p,q} := p^{\text{max}_{\chi \in \Omega_E} \left( \text{rk}(\chi) \right)}.
\end{equation}
Then there exists a fixed constant $C_{E,p,q} > 0$ which depends only on the elliptic curve $E$, $p$, and $q$ such that for every $d > \frac{12 \log p + 2\log D_{E,p,q} + (6 \log p) \cdot \# \Sigma(\sigma)}{\log q}$,
\begin{equation}
    \left| \frac{\# \{ \mathfrak{v} \in \mathcal{P}_i(d) \; | \; \mathfrak{v} \not\in \Sigma(\sigma) \text{ and } t_\chi(\mathfrak{v}) = j \}}{\# \{\mathfrak{v} \in \mathcal{P}_i(d) \; | \; \mathfrak{v} \not\in \Sigma(\sigma)\}} - d_{i,j} \right| < C_{E,p,q} \cdot p^{3 \# \Sigma(\sigma)} \cdot q^{-\frac{d}{2}}.
\end{equation}
\end{proposition}
\begin{proof}
The theorem can be proved in an analogous way to how \cite[Proposition 9.4]{KMR14} was proved over number fields. Nevertheless, it is necessary to apply the effective Chebotarev density theorem to calculate the explicit error bounds.

\medskip
\textbf{[[Governing field extension for $t_\chi(\mathfrak{v})$]]}
\medskip

We first review the ideas presented in \cite[Proposition 9.4]{KMR14}. Denote by $Res$ the restriction morphism of cohomology groups:
\begin{equation*}
    H^1_{\et}(K, E[p]) \to H^1_{\et}(K(E[p]), E[p])^{\Gal(K(E[p])/K)} = \text{Hom}(\Gal(\overline{K(E[p])}/K(E[p])), E[p])^{\Gal(K(E[p])/K)}.
\end{equation*}
Let $F_{\sigma, \chi}$ be the fixed field of the following subgroup of $\Gal(\overline{K(E[p])}/K(E[p]))$:
\begin{equation*}
    \bigcap_{c \in \Sel(E[p],\chi)} \text{Ker} \left( Res(c): \Gal(\overline{K(E[p])}/K(E[p])) \to E[p] \right).
\end{equation*}
The field $F_{\sigma, \chi}$ satisfies the following properties, as shown in \cite[Proposition 9.3]{KMR14}:
\begin{enumerate}
    \item $F_{\sigma, \chi}$ is Galois over $K$.
    \item There is a $\Gal(K(E[p])/K)$-module isomorphism $\Gal(F_{\sigma, \chi}/K(E[p])) \cong (E[p])^{\text{rk}(\chi)}$.
    \item $F_{\sigma, \chi}/K$ is unramified outside of places in $\Sigma(\sigma)$.
\end{enumerate}
The aforementioned condition holds for $p=2$ whenever $E$ is a non-isotrivial elliptic curve such that $\Gal(K(E[2])/K) \cong S_3$.

\medskip
\textbf{[[Constant field of $F_{\sigma, \chi}$]]}
\medskip

Suppose that $E$ has a place $v$ of split multiplicative reduction. Then the constant field of $F_{\sigma, \chi}$ is equal to $\F_q$. It suffices to show that any basis element $c \in \Sel(E[p],\chi)$ maps the arithmetic Frobenius $\tau \in \Gal(\overline{\F_q}/\F_q)$ to the identity element of $E[p]$. Consider the local Kummer map $im \delta_v^\chi$ at the place $v$.Then $E$ is a Tate curve at $v$. There exists an element $q \in K_v^\times$ with positive valuation such that the $\overline{K_v}$-rational points of $E$ is given by
\begin{equation*}
    E(\overline{K_v}) \cong \overline{K_v}^\times/\langle q \rangle,
\end{equation*}
which implies for any positive number $n$,
\begin{equation*}
    E[n](\overline{K_v}) \cong \langle q^{\frac{1}{n}}, \mu_n \rangle / \langle q \rangle,
\end{equation*}
see for example [Section 3.3]\cite{BLV09} for a detailed discussion on these results. To analyze the condition that the basis element $c \in \Sel(E[p],\chi)$ maps the arithmetic Frobenius $\tau \in \Gal(\overline{\F_q}/\F_q)$ to the identity element of $E[p]$, it suffices to verify that $Q^\tau - Q = O$ for $Q \in E[p](\overline{K_v})$, which follows from the assumption that the constant field of ${K_v}$ contains the primitive $p$th-root of unity.

\medskip
\textbf{[[Frobenius conjugacy class]]}
\medskip

Using the techniques of the proof from \cite[Proposition 9.4]{KMR14}, one can show that the non-zero values of $d_{i,j}$ from the table of the statement of the proposition are ratios of two non-empty subsets $S_{i,j}, S'_{i} \subset \Gal(F_{\sigma, \chi}/K)$ stable under conjugation, i.e. $d_{i,j} = \frac{\# S_{i,j}}{\# S'_{i}}$. These subsets satisfy the condition that
\begin{equation}
    \begin{cases}
    \mathfrak{v} \in \mathcal{P}_i(d) &\Longleftrightarrow \Frob_\mathfrak{v} \in S'_i, \\
    \dim_{\F_p} \text{im} {\delta}_\mathfrak{v}^\chi = j \text{ and } \mathfrak{v} \in \mathcal{P}_i(d) &\Longleftrightarrow \Frob_\mathfrak{v} \in S_{i,j}.
    \end{cases}
\end{equation}
We refer to \cite[Proposition 9.4]{KMR14} for a detailed description of what these subsets are in $\Gal(F_{\sigma, \chi}/K)$. 

\medskip
\textbf{[[Effective error bounds]]}
\medskip

Because the constant field of $F_{\sigma, \chi}$ is $\F_q$, we can use Theorem \ref{theorem:effective_chebotarev} to bound the error terms of the following equation:
\begin{equation}
    \left| \frac{\# \{ \mathfrak{v} \in \mathcal{P}_i(d) \; \mid \; \mathfrak{v} \not\in \Sigma(\sigma) \text{ and } t_\chi(\mathfrak{v}) = j \}}{\# \{ \mathfrak{v} \in \mathcal{P}_i(d) \; \mid \; \mathfrak{v} \not\in \Sigma(\sigma) \}} - d_{i,j} \right|.
\end{equation}

To apply Theorem \ref{theorem:effective_chebotarev}, one needs to understand how the groups $G$ as well as the genus $g_{F_{\sigma,\chi}}$ grow in terms of $\deg \sigma$. Recall that $D_{E,p,q} > 0$ is a constant defined as
\begin{equation}
    D_{E,p,q} := p^{\text{max}_{\chi \in \Omega_E} \left( \text{rk}(\chi) \right)}.
\end{equation} 
Proposition \ref{prop:Markov_equi} shows that
\begin{equation}
    \# \Gal(F_{\sigma,\chi}/K) = [F_{\sigma,\chi} : K(E[p])] \leq D_{E,p,q} \cdot p^{2 \# \Sigma(\sigma)} \cdot (p^3 - p)
\end{equation}
is a constant that only depends on the choice of the elliptic curve $E$, $q$, and $p$. Recall that $F_{\sigma, \chi}/K$ is unramified away from $v \in \Sigma(\sigma)$. Hence, the Riemann-Hurwitz theorem implies that
\begin{equation*}
    g_{F_{\sigma, \chi}} \leq D_{E,p,q} \cdot p^{2 \# \Sigma(\sigma)} \cdot (p^3-p) \cdot \# \Sigma(\sigma).
\end{equation*}
Then one obtains that
\begin{align}
\begin{split}
    \# \Gal(F_{\sigma,\chi}/K) + g_{F_{\sigma,\chi}} &\leq D_{E,p,q} \cdot p^{2 \# \Sigma(\sigma)}  \cdot (p^3 - p) \cdot (1 + \#\Sigma(\sigma)) \\
    &\leq D_{E,p,q} \cdot p^{2 \# \Sigma(\sigma) + 4} \cdot \# \Sigma(\sigma) \\
    &\leq D_{E,p,q} \cdot p^{3 \# \Sigma(\sigma) + 4}.
\end{split}
\end{align}
Corollary \ref{corollary:effective_chebotarev} implies that for any $d$ satisfying
\begin{equation} \label{eq:d_technical_condition}
    d > \frac{12 \log p + 2\log D_{E,p,q} + (6 \log p) \cdot \# \Sigma(\sigma)}{\log q}
\end{equation}
the following inequality holds:
\begin{equation*}
    \left| \frac{\# \{ \mathfrak{v} \in \mathcal{P}_i(d) \; \mid \; \mathfrak{v} \not\in \Sigma(\sigma) \text{ and } t_\chi(\mathfrak{v}) = j \}}{\# \{ \mathfrak{v} \in \mathcal{P}_i(d) \; \mid \; \mathfrak{v} \not\in \Sigma(\sigma) \}} - d_{i,j} \right| < 16 \cdot D_{E,p,q} \cdot p^{3 \# \Sigma(\sigma) + 4} \cdot q^{-\frac{d}{2}}.
\end{equation*}
Letting $C_{E,p,q} = 16 \cdot D_{E,p,q} \cdot p^4$ proves the statement of the theorem.
\end{proof}

\begin{remark}
The technical condition on the degree of the place $\mathfrak{v}$ will be used in the upcoming sections when we compute the probability distribution of local Selmer ranks of elliptic curves twisted by cyclic order-$p$ characters associated to $p$-th power free polynomials $f$ of large enough degree $n$. We will show that for almost all $f \in F_n(\F_q)$, the cardinality of the associated set $\Sigma(\sigma)$ is bounded above by $2 m_{n,q} := 2(\log n + \log \log q)$ by Theorem \ref{theorem:erdos_kac}. This in turn will allow us to compute the probability distribution of $\pi$-Selmer rank of the cyclic order-$p$ twists of $E$ from local Selmer ranks $\Sel(E[p],\chi)$.
\end{remark}

\begin{remark}
Proposition \ref{theorem:local_twists} states that if $\Gal(K(E[p])/K) \supset \SL_2(\F_p)$, then the Chebotarev density theorem completely determines the variations of $\pi$-Selmer groups of elliptic curves twisted by local cyclic order-$p$ characters. This is not the case if the Galois group $\Gal(K(E[p])/K)$ does not contain $\SL_2(\F_p)$, as carefully studied in \cite{FIMR13} and \cite{Sm22_01}. For example, suppose that $p = 2$ and $\Gal(K(E[p])/K) = \mathbb{Z}/3\mathbb{Z}$. Friedlander, Iwaniec, Mazur, and Rubin showed that the variation of 2-Selmer groups of certain subfamilies of quadratic twists of elliptic curves are governed by the spin of odd principal prime ideals defined over totally real cyclic Galois extensions \cite[Chapter 3, Chapter 10]{FIMR13}. Smith uses a generalized notion of spin of prime ideals called ``symbols of prime ideals" \cite[Definition 3.11, Proposition 3.14]{Sm22_01} to classify which classes of prime ideals equivalently varies the Selmer groups of twistable modules, a generalized notion of quadratic twist families of abelian varieties \cite[Chapter 4]{Sm22_01}. Thankfully, Proposition \ref{theorem:local_twists} demonstrates that one does not require to use the spin of prime ideals to determine the variations of the dimensions of $\Sel(E[p],\chi)$ as $\chi$ varies over the set of Cartesian product of local characters.
\end{remark}

\subsection{Auxiliary places}
\label{section:auxiliary_places}

Given a polynomial $f \in F_n(\F_q)$, recall from the introduction that we can identify a cyclic order-$p$ character $\chi_f \in \text{Hom}(\text{Gal}(\overline{K}/K), \mu_p)$ via the quotient map
\begin{equation*}
    \chi_f: \text{Gal}(\overline{K}/K) \twoheadrightarrow \text{Gal}(L^f/K) \to \mu_p
\end{equation*}
that maps the generator $\sigma_f \in \text{Gal}(L^f/K)$ to $\zeta_p$. Given a place $v$ of $K$, denote by $\chi_{f,v} \in \text{Hom}(\text{Gal}(\overline{K}_v/K_v), \mu_p)$ the restriction of the global character $\chi_f$ to $K_v$.

The goal of this subsection is to understand the distribution of $\text{rk}((\chi_{f,v})_v)$ as $f$ ranges over the set $F_{(n,N),(w,w')}^{(\lambda,\eta)}(\F_q)$ for some $\lambda \in \Lambda_{N,w'}^{la}$ and $\eta \in \Lambda_{n-N,w-w'}^{for}$. To do so, we introduce the notion of an auxiliary place of a polynomial $f \in F_{(n,N),(w,w')}^{(\lambda,\eta)}(\F_q)$.

\begin{definition}
Let $f \in F_n(\F_q)$. Denote by $\overline{f}$, $\overline{f}_*$, and $\overline{f}^*$ the square-free polynomial over $\F_q$ defined as
\begin{align}
\begin{split}
    \overline{f} := \prod_{\substack{g \mid f \\ g \in \mathcal{P}_1 \cup \mathcal{P}_2}} g, \; \; \overline{f}_* := \prod_{\substack{g \mid f_* \\ g \in \mathcal{P}_1 \cup \mathcal{P}_2}} g, \; \; \overline{f}^* := \prod_{\substack{g \mid f^* \\ g \in \mathcal{P}_1 \cup \mathcal{P}_2}} g
\end{split}
\end{align}
i.e. they are products of irreducible factors of $f$ (and $f_*$ and $f^*$, respectively) of degree greater than $\mathfrak{n}$ which lies in $\mathcal{P}_1$ or $\mathcal{P}_2$.
\end{definition}

\begin{definition}[Auxiliary place] \label{defn:aux}
    Given positive integers $n > N$ and $w > w'$, let $\lambda \in \Lambda_{N,w'}^{la}$ and $\eta \in \Lambda_{n-N,w-w'}^{for}$ be splitting partitions.
    \begin{itemize}
    \item Given a degree $n$ polynomial $f \in F_{(n,N),(w,w')}^{(\lambda,\eta)}(\F_q)$, an auxiliary place of $f$ is an irreducible polynomial $g \in \mathcal{P}_0$ of maximal degree dividing $f$. We denote by $d_a$ the degree of an auxiliary place of any $f \in F_{(n,N),(w,w')}^{(\lambda,\eta)}(\F_q)$. In particular, $d_a$ is the maximal degree that guarantees $\lambda_{i,j,0} = 0$ for every $i > d_a$.
    
    \item We denote by $f_a$ the auxiliary factor of $f$ defined as
    \begin{equation}
        f_a := \prod_{\substack{g \mid f \\ g \in \mathcal{P}_0(d_a)}} g^{v_g(f)}.
    \end{equation}
    It is the product of all auxiliary places of $f$.

    \item We denote by $d_{a^*}$ the degree of the auxiliary factor of $f$, which can be written as
    \begin{equation}
        d_{a^*} := d_a \cdot \left(\sum_{j=1}^{p-1} \lambda_{d_a,j,0} \right).
    \end{equation}

    \item Fix a polynomial $h \in F_{n- d_{a^*}}(\F_q)$. We define the following subset of $F_{(n,N),(w,w')}^{(\lambda,\eta)}(\F_q)$:
    \begin{equation}
        F_{(n,N),(w,w')}^{(\lambda,\eta), h}(\F_q) := \left\{f \in F_{(n,N),(w,w')}^{(\lambda,\eta)}(\F_q) \; | \; \frac{f}{f_a} = h\right\}.
    \end{equation}
    The above subset is empty if $h$ does not divide any polynomial in $F_{(n,N),(w,w')}^{(\lambda,\eta)}(\F_q)$. By definition, the following relation holds:
    \begin{equation}
        F_{(n,N),(w,w')}^{(\lambda,\eta)}(\F_q) = \bigsqcup_{h \in F_{n- d_{a^*}}(\F_q)} F_{(n,N),(w,w')}^{(\lambda,\eta),h}(\F_q).
    \end{equation}
    \end{itemize}
\end{definition}

\begin{definition}
    Let $f \in F_n(\F_q)$. We denote by $\Sigma_f$ the set of places
    \begin{equation}
        \Sigma_f := \Sigma_E \cup \{v \in \mathcal{P} \; | \; v \text{ divides } f_* \}.
    \end{equation}
    We note that if $f \in F_{(n,N),(w,w')}^{(\lambda,\eta)}$, then $\# \Sigma_f = \# \Sigma_E + (w-w')$.
\end{definition}

\begin{definition}
    Given a polynomial $f \in F_{(n,N),(w,w')}^{(\lambda,\eta)}(\F_q)$, we use the abbreviation $\Omega_{\overline{f}^*}$ to denote the set of finite Cartesian products of local characters
    \begin{align}
    \begin{split}
        \Omega_1 &= \prod_{v \in \Sigma_f} \text{Hom}(\text{Gal}(\overline{K}_v/K_v),\mu_p), \\
        \Omega_{\overline{f}^*} &= \prod_{v \in \Sigma_f} \text{Hom}(\text{Gal}(\overline{K}_v/K_v),\mu_p) \times \prod_{\substack{v \mid f^* \\ v \nmid f_a}} \text{Hom}_{ram}(\text{Gal}(\overline{K}_v/K_v),\mu_p),
    \end{split}
    \end{align}
    such that the component $\chi_v$ is ramified if $v \mid f^*$, and we ignore the local characters at any places $v$ dividing the auxiliary factor $f_a$ of $f$. In particular, we enlarge the set $\Sigma$ from Definition \ref{definition:basic_local_def} to include places $v \mid f_*$ and set $\Sigma = \Sigma_f$, even though $\chi_{f,v}$ is ramified at such places.
\end{definition}
In order to make this reformulation more concrete, we present an alternative way to define the subset $F_{(n,N),(w,w')}^{(\lambda,\eta)}(\F_q)$ given partitions $\lambda := \{(\lambda_{i,j,k},i,j,k)\} \in \Lambda_{N,w'}^{ad}$ and $\eta := \{(\eta_{\hat{i},\hat{j},\hat{k}},\hat{i},\hat{j},\hat{k})\} \in \Lambda_{n-N,w-w'}^{for}$. Given a set $X$, we denote by
\begin{equation}
    \PConf_n(X) := \{(x_1, \cdots, x_n) \in X^{\oplus n} \; | \; x_i \neq x_j \text{ for all } 1 \leq i < j \leq n\}
\end{equation}
the set-theoretic ordered configuration set of $n$ elements in $X$. There is a transitive action of the symmetric group $S_n$ on $\PConf_n(X)$, which prompts us to define
\begin{equation}
    \Conf_n(X) := \PConf_n(X) / S_n
\end{equation}
the set-theoretic unordered configuration set of $n$ elements in $X$. Using these notations, we can define the subset $F_{(n,N),(w,w')}^{(\lambda,\eta)}(\F_q)$ as
\begin{equation}
    F_{(n,N),(w,w')}^{(\lambda,\eta)}(\F_q) := \left[ \prod_{i,j,k} \Conf_{\lambda_{i,j,k}}(\mathcal{P}_k(i)) \right] \times \left[ \prod_{\hat{i},\hat{j},\hat{k}} \Conf_{\eta_{\hat{i},\hat{j},\hat{k}}}(\mathcal{P}_{\hat{k}}(\hat{i})) \right]
\end{equation}
where we regard $\Conf_0(X) = \{0\}$. In particular, if a polynomial $f \in F_{(n,N),(w,w')}^{(\lambda,\eta)}(\F_q)$ admits an irreducible factorization via
\begin{align}
    \begin{split}
        f^* &:= \prod_{i,j,k} \prod_{m=1}^{\lambda_{i,j,k}} g_{i,j,k,m}^j, \; \; \; \; 
        f_* := \prod_{\hat{i},\hat{j},\hat{k}} \prod_{m=1}^{\eta_{\hat{i},\hat{j},\hat{k}}} h_{\hat{i},\hat{j},\hat{k},m}^{\hat{j}}
    \end{split}
\end{align}
where $\{g_{i,j,k,m}\}$ and $\{h_{\hat{i},\hat{j},\hat{k},m}\}$ are sets of irreducible factors of $f$, then under this identification a polynomial $f \in F_{(n,N),(w,w')}^{(\lambda,\eta)}(\F_q)$ can be represented as an element
\begin{equation}
    \left(\prod_{i,j,k} \{g_{i,j,k,m}\}_{m=1}^{\lambda_{i,j,k}} \right) \times \left(\prod_{\hat{i},\hat{j},\hat{k}} \{h_{\hat{i},\hat{j},\hat{k},m}\}_{m=1}^{\eta_{\hat{i},\hat{j},\hat{k}}} \right).
\end{equation}

Using this identification, we can reformulate Definition \ref{defn:aux} as follows. There is a natural projection map
\begin{align*}
    \phi_{d_a}: &\left[ \prod_{i,j,k} \Conf_{\lambda_{i,j,k}}(\mathcal{P}_k(i)) \right] \times \left[ \prod_{\hat{i},\hat{j},\hat{k}} \Conf_{\eta_{\hat{i},\hat{j},\hat{k}}}(\mathcal{P}_{\hat{k}}(\hat{i})) \right] \\
    &\to \left[ \prod_{\substack{i,j,k \\ (i,k) \neq (d_a,0)}} \Conf_{\lambda_{i,j,k}}(\mathcal{P}_k(i)) \right] \times \left[ \prod_{\hat{i},\hat{j},\hat{k}} \Conf_{\eta_{\hat{i},\hat{j},\hat{k}}}(\mathcal{P}_{\hat{k}}(\hat{i})) \right]
\end{align*}
which forgets all the irreducible factors of $f \in F_{(n,N),(w,w')}^{(\lambda,\eta)}(\F_q)$ lying in $\prod_{j=1}^{p-1}\Conf_{\lambda_{d_a,j,0}}(\mathcal{P}_0(d_a))$. Then
\begin{equation}
    F_{(n,N),(w,w')}^{(\lambda,\eta),h}(\F_q) = \phi_{d_a}^{-1}(h).
\end{equation}
where $h \in F_{n-d_{a^*}}(\F_q)$ such that $h \mid f$ for some $f \in F_{(n,N),(w,w')}^{(\lambda,\eta)}(\F_q)$.

Using the notations introduced in this subsection, an immediate result of Corollary \ref{corollary:P0_quadratic_character_sum} can be stated as follows.
\begin{corollary} \label{corollary:global_character_sum}
    Fix a locally arrangeable partition $\lambda \in \Lambda_{N,w'}^{la}$ and a forgettable partition $\eta \in \Lambda_{n-N, w-w'}^{for}$.
    Fix a polynomial $h \in F_{n- d_{a^*}}(\F_q)$. Suppose the set $F_{(n,N),(w,w')}^{(\lambda,\eta), h}(\F_q) = \phi_{d_a}^{-1}(h)$ is non-empty and $w \leq 2m_{n,q}$. Denote by $h_1, h_2, \cdots, h_{w(h)}$ the irreducible factors of $h$. Denote by $\hat{w}(h)$ the quantity
    \begin{equation*}
        \hat{w}(h) := \begin{cases}
            w(h) &\text{ if } K(\sqrt[p]{h_1}, \cdots, \sqrt[p]{h_{w(h)}}) \cap K(E[p]) = K \text{ or } p \geq 5, \\
            w(h) - 1 &\text{ if } K(\sqrt[p]{h_1}, \cdots, \sqrt[p]{h_{w(h)}}) \cap K(E[p]) \neq K \text{ and } p \leq 3.
        \end{cases}
    \end{equation*}
    
    Let $\chi := (\chi_v)_v \in \Omega_{\overline{f}^*}$ be any product of local characters, whose components are ramified at places $v \mid h$ and unramified elsewhere such that there exists $f \in \phi^{-1}_{d_a}(h)$ such that $\text{Ker}(\chi_{f,v}) = \text{Ker}(\chi_v)$ for all $v \in  \Sigma(\overline{h})$. Then we have
    \begin{equation*}
        \left|\frac{\# \{f \in \phi_{d_a}^{-1}(h) \; | \; \text{Ker}(\chi_{f,v}) = \text{Ker}(\chi_v) \; \forall \; v \in  \Sigma(\overline{h})\}}{\# \phi_{d_a}^{-1}(h)} - \frac{1}{p^{\hat{w}(h)}} \right| < \hat{C}_{E,p,q} \cdot (n\log q)^{-2m_{n,q} + 2 \log p},
    \end{equation*}
    where $\hat{C}_{E,p,q} > 0$ is the constant introduced from Corollary \ref{corollary:P0_quadratic_character_sum}.
\end{corollary}
The condition that $\text{Ker}(\chi_{f,v}) = \text{Ker}(\chi_v)$ as subgroups of $\text{Gal}(\overline{K}_v/K_v)$ for each place $v \in \Sigma(\overline{h})$ implies that the fixed fields $\overline{K}_v^{\text{Ker}(\chi_{f,v})}$ and $\overline{K}_v^{\text{Ker}(\chi_{v})}$, which are cyclic ramified extensions of degree $p$ over $K_v$, are equal to each other.
\begin{proof}
    We note that there exists a bijection between the following sets:
    \begin{align}
        \begin{split}
            \phi_{d_a}^{-1}(h) &\to \prod_{j=1}^{p-1} \Conf_{\lambda_{d_a,j,0}},(\mathcal{P}_{0}(d_a)) \\
            f = h f_a &\mapsto f_a.
        \end{split}
    \end{align}
    There is an $\prod_{j=1}^{p-1} S_{\lambda_{d_a,j,0}}$-equivariant covering map
    \begin{align}
        F: \prod_{j=1}^{p-1} \PConf_{\lambda_{d_a,j,0}}(\mathcal{P}_{0}(d_a)) \to \prod_{j=1}^{p-1} \Conf_{\lambda_{d_a,j,0}}(\mathcal{P}_{0}(d_a)),
    \end{align}
    where for any fixed $f_a$, every element in $F^{-1}(f_a)$ restricts to an identical character in $\Omega_{\overline{f}^*}$. It hence suffices to compute the desired probability over the ordered configuration set $\PConf_{\lambda_{d_a,j,0}}(\mathcal{P}_{0}(d_a))$. This can be achieved by applying Corollary \ref{corollary:P0_quadratic_character_sum} and using the fact that every ramified cyclic degree $p$ extension of $K_v$ is obtained from adjoining to $K_v$ the $p$-th roots of elements of form $\pi_v u^i$, where $\pi_v$ is a uniformizer of $K_v$, $u \in K_v^\times/(K_v^{\times})^p$ is non-trivial, and $0 \leq i \leq p-1$.
\end{proof}

\begin{definition} \label{defn:technical_phi}
    Given a locally arrangeable partition $\lambda \in \Lambda_{N,w'}^{la}$ and a forgettable partition $\eta \in \Lambda_{n-N,w-w'}^{for}$, consider the set of polynomials $F_{(n,N),(w,w')}^{(\lambda,\eta)}(\F_q)$. 
    
    Fix $1 \leq j^* \leq p-1$ and $0 \leq k^* \leq 2$. Let $d$ be an integer such that $d \neq d_a$ and $\lambda_{d,j^*,k^*} \neq 0$.
    \begin{enumerate}
    \item We denote by $\phi_{d,j^*,k^*}$ the canonical projection map
    \begin{equation*}
    \phi_{d,j^*,k^*}: F_{(n,N),(w,w')}^{(\lambda,\eta)}(\F_q) \to \left[ \prod_{\substack{i,j,k \\ (i,k) \neq (d_a,0) \\ (i,j,k) \neq (d,j^*,k^*)}} \Conf_{\lambda_{i,j,k}}(\mathcal{P}_k(i)) \right] \times \left[ \prod_{\hat{i},\hat{j},\hat{k}} \Conf_{\eta_{\hat{i},\hat{j},\hat{k}}}(\mathcal{P}_{\hat{k}}(\hat{i})) \right]
    \end{equation*}
    which forgets the irreducible factors of $f \in F_{(n,N),(w,w')}^{(\lambda,\eta)}(\F_q)$ lying in the set
    \begin{equation*}
        \Conf_{\lambda_{d,j^*,k^*}}(\mathcal{P}_{k^*}(d)) \times \prod_{j=1}^{p-1} \Conf_{\lambda_{d_a,j,0}}(\mathcal{P}_0(d_a)).
    \end{equation*}
    \item Denote by $D := n - d_{a^*} - d \cdot j^* \cdot \lambda_{d,j^*,k^*}$. Let $h \in F_D(\F_q)$ be a polynomial such that $h \mid f$ for some $f \in F_{(n,N),(w,w')}^{(\lambda,\eta)}(\F_q)$. Denote by $\phi_{d,j^*,k^*}^{-1}(h) \subset F_{(n,N),(w,w')}^{(\lambda,\eta)}(\F_q)$ the set of fibers of $\phi_{d,j^*,k^*}$ at $h$. This set admits the following bijection:
    \begin{equation*}
        \phi_{d,j^*,k^*}^{-1}(h) \cong \Conf_{\lambda_{d,j^*,k^*}}(\mathcal{P}_{k^*}(d)) \times \prod_{j=1}^{p-1} \Conf_{\lambda_{d_a,j,0}}(\mathcal{P}_0(d_a)).
    \end{equation*}
    \end{enumerate}
\end{definition}

The upcoming proposition combines equidistribution of characters from Corollary \ref{corollary:global_character_sum} and the Chebotarev conditions from Proposition \ref{prop:Markov_equi} and Proposition \ref{theorem:local_twists}. This allows us to obtain the distribution of changes in dimensions of local Selmer groups of $E$ associated to consecutive twists of local characters.
\begin{proposition} \label{prop:local_twist_2}
    Assume the notations and conditions as stated in Definition \ref{defn:technical_phi}. Let $E/K$ be an elliptic curve satisfying conditions in (\ref{equation:assumption_local_twists}). 
    
    Given $f \in \phi_{d,j^*,k^*}^{-1}(h)$, let $\omega_f$ and $\omega_f'$ be defined as
        \begin{equation}
                \omega_f := (\chi_{f,v})_{v \in \Sigma_f(\overline{h}^*)} \in \Omega_{\overline{h}^*}, \; \; \; \; 
                \omega_f' := (\chi_{f,v})_{v \in \Sigma_f(\overline{f}^*)} \in \Omega_{\overline{f}^*}.
        \end{equation}
    Denote by $\delta_h: \mathbb{Z}_{\geq 0} \to [0,1]$ the probability distribution
    \begin{equation}
        \delta_h(J) := \frac{\#\{f \in \phi_{d,j^*,k^*}^{-1}(h) \; | \; \text{rk}(\omega_f) = J\}}{\# \phi_{d,j^*,k^*}^{-1}(h)}.
    \end{equation}
    Let $\tilde{k} := \lambda_{d,j^*,k^*} \cdot k^*$. Then for any $n$ such that $m_{n,q} > \max\{\deg \Delta_E, 3 \cdot \log p\}$, there exists a fixed constant $B_{E,p,q} > 0$ dependent only on $E,p,q$ such that
        \begin{equation}
            \left|\frac{\#\{f \in \phi_{d,j^*,k^*}^{-1}(h) \; | \; \text{rk}(\omega_f') = J\}}{\# \phi_{d,j^*,k^*}^{-1}(h)} - (M_L^{\tilde{k}} \delta_h)(J) \right| < \lambda_{d,j^*,k^*} \cdot B_{E,p,q} \cdot (n \log q)^{-2m_{n,q} + 6 \log p + 1}.
        \end{equation}
        where $M_L := [p_{r,s}]$ is the Markov operator over $\mathbb{Z}_{\geq 0}$ given by
    \begin{equation*}
    p_{r,s} = \begin{cases}
    1 - p^{-r} &\text{ if } s = r-1 \geq 0, \\
    p^{-r} &\text{ if } s = r+1, \\
    0 &\text{ else}.
    \end{cases}
\end{equation*}
\end{proposition}
\begin{proof}
    Definition \ref{defn:technical_phi} implies that
    \begin{equation*}
        \phi^{-1}_{d,j^*,k^*}(h) = \Conf_{\lambda_{d,j^*,k^*}}(\mathcal{P}_{k^*}(d)) \times \prod_{j=1}^{p-1} \Conf_{\lambda_{d_a,j,0}}(\mathcal{P}_0(d_a)).
    \end{equation*}
    Throughout the proof of this proposition, we use the index $\ell$ to denote the coordinates of the elements $(g_1, g_2, \cdots, g_{\lambda_{d,j^*,k^*}}) \in \Conf_{\lambda_{d,j^*,k^*}}(\mathcal{P}_{k^*}(d))$.
    
    By Corollary \ref{corollary:global_character_sum}, and the condition that $w \leq 2 m_{n,q}$, for any fixed $\lambda_{d,j^*,k^*}$ many distinct elements $g_1, g_2, \cdots, g_{\lambda_{d,j^*,k^*}} \in \mathcal{P}_{k^*}(d)$ and any $\omega := (\omega_v)_v \in \Omega_{\overline{f}^*}$, there exists an explicit constant $\hat{C}_{E,p,q} > 0$ such that
    \begin{align} \label{eqn:prop5.21eq1}
    \begin{split}
        & \; \; \; \; \left| \frac{\# \{f \in \phi_{d,j^*,k^*}^{-1}(h) \; | \; \frac{f}{f_a \cdot h} = \prod_{\ell=1}^{\lambda_{d,j^*,k^*}} g_\ell^{j^*}, \; \text{Ker}(\chi_{f,g_\ell}) = \text{Ker}(\omega_{g_\ell}) \; \forall \; \ell \}}{\# \phi_{d_a}^{-1}(h)} - \frac{1}{p^{\lambda_{d,j^*,k^*}}} \right| \\
        &< \hat{C}_{E,p,q} \cdot (n\log q)^{-2m_{n,q} + 2 \log p}.
    \end{split}
    \end{align}

    \medskip
    \textbf{[From global statistics to local statistics]}
    \medskip 

    The goal of this subsection of the proof is to demonstrate that the statistical statement on local Selmer structures parametrized by polynomials $f \in \phi^{-1}_{d,j^*,k^*}(h)$ can be reduced to the statistical statement on local Selmer structures parametrized by subsets of Cartesian products of local characters in $\Omega_{\overline{f}^*}$. Given two non-negative integers $J_0$ and $J_1$, we note that
    \begin{small}
    \begin{align} \label{eqn:prop5.21eq3}
        \begin{split}
            & \; \; \; \; \# \{f \in \phi^{-1}_{d,j^*,k^*}(h) \; | \; \text{rk}(\omega_f') = J_1, \text{rk}(\omega_f) = J_0\} \\
            &= \sum_{(g_\ell)_\ell \in \Conf_{\lambda_{d,j^*,k^*}}(\mathcal{P}_{k^*}(d))} \# \left\{f \in \phi^{-1}_{d,j^*,k^*}(h) \; | \; \frac{f}{f_a \cdot h} = \prod_{\ell=1}^{\lambda_{d,j^*,k^*}} g_\ell^{j^*}, \text{rk}(\omega_f') = J_1, \text{rk}(\omega_f) = J_0\right\}.
        \end{split}
    \end{align}
    \end{small}
    Each summand
    \begin{equation} \label{eqn:prop5.21eq3a}
        \# \left\{f \in \phi^{-1}_{d,j^*,k^*}(h) \; | \; \frac{f}{f_a \cdot h} = \prod_{\ell=1}^{\lambda_{d,j^*,k^*}} g_\ell^{j^*}, \text{rk}(\omega_f') = J_1, \text{rk}(\omega_f) = J_0\right\}
    \end{equation}
    can be evaluated as
    \begin{equation*}
        = \begin{cases}
                \# \phi_{d_a}^{-1}(h) \cdot  \delta_{h,(g_\ell)_\ell}(J_0) &\text{ if } \text{rk}(\omega_f') - \text{rk}(\omega_f) = J_1 - J_0, \\
                0 &\text{ otherwise },
                \end{cases}
    \end{equation*}
     where $\delta_{h,(g_\ell)_\ell}: \mathbb{Z}_{\geq 0} \to [0,1]$ is a probability distribution defined as
    \begin{equation}
        \delta_{h,(g_\ell)_\ell}(J) := \frac{\# \left\{f \in \phi^{-1}_{d,j^*,k^*}(h) \; | \; \frac{f}{f_a \cdot h} = \prod_{\ell=1}^{\lambda_{d,j^*,k^*}} g_\ell^{j^*}, \text{rk}(\omega_f) = J\right\}}{\# \phi^{-1}_{d_a}(h)}.
    \end{equation}
    By definition, the following equation holds for all $J \in \mathbb{Z}_{\geq 0}$:
    \begin{equation} \label{eqn:prop5.21eq3aa}
        \delta_h(J) = \frac{1}{\# \Conf_{\lambda_{d,j^*,k^*}}(\mathcal{P}_{k^*}(d))} \sum_{(g_\ell)_\ell \in \Conf_{\lambda_{d,j^*,k^*}}(\mathcal{P}_{k^*}(d))}\delta_{h,(g_\ell)_\ell}(J).
    \end{equation}
    Suppose we have two products of local characters $\omega = (\omega_v)_v$ and $\omega' = (\omega'_v)_v$ in $\Omega_{\overline{f}^*}$. The definitions of local Selmer groups $\Sel(E[p],\omega)$ and $\Sel(E[p],\omega')$ and the fact that the local conditions at $v \in \mathcal{P}_0$ do not affect the dimensions of local Selmer groups imply that if $\text{Ker}(\omega_v) = \text{Ker}(\omega'_v)$ for all $v \in \Sigma_f(\overline{h}^*)$, then
    \begin{equation} \label{eqn:prop5.21eq3aaa}
        \text{rk}((\omega_v)_{v \in \Sigma_f(\overline{h}^*)}) = \text{rk}((\omega'_v)_{v \in \Sigma_f(\overline{h}^*)}).
    \end{equation}
    (And if in addition $\text{Ker}(\omega_{g_\ell}) = \text{Ker}(\omega'_{g_\ell})$ for all irreducible elements $g_\ell$ of $(g_\ell)_\ell \in \Conf_{\lambda_{d,j^*,k^*}}(\mathcal{P}_{k^*}(d))$, then we can further guarantee that the dimensions of $\Sel(E[p],\omega)$ and $\Sel(E[p],\omega')$ are equal to each other.)
    
    Equation (\ref{eqn:prop5.21eq3aaa}) and Corollary \ref{corollary:global_character_sum} imply that for any $(g_\ell)_\ell \in \Conf_{\lambda_{d,j^*,k^*}}(\mathcal{P}_{k^*}(d))$,
    \begin{equation} \label{eqn:prop5.21eq3ab}
        \sup_{J \in \mathbb{Z}_{\geq 0}}|\delta_h(J) - \delta_{h,(g_\ell)_\ell}(J)| < \hat{C}_{E,p,q} \cdot (n\log q)^{-2 m_{n,q} + 4 \log p}.
    \end{equation}
    We note that the exponent for $n\log q$ changes from $-2m_{n,q} + 2 \log p$ to $-2m_{n,q} + 4 \log p$ because there are at most $p^{2m_{n,q}} \leq (n \log q)^{2 \log p}$ many ramified cyclic extensions over local fields one needs to consider to determine the dimensions of local Selmer groups.

    Given an element $(g_\ell)_\ell \in \Conf_{\lambda_{d,j^*,k^*}}(\mathcal{P}_{k^*}(d))$, we use the abbreviation $g$ to denote the following square-free polynomial over $\mathbb{F}_q$:
    \begin{equation*}
        g := \prod_{\ell=1}^{\lambda_{d,j^*,k^*}} g_\ell.
    \end{equation*}
    We denote by $\Omega_{\omega_f,g}$ the subset of local characters $\chi' \in \Omega_{\overline{h}^* \cdot g} = \Omega_{\overline{f}^*}$ satisfying the two conditions below:
    \begin{itemize}
        \item For any $v \in \Sigma(\overline{h}^*)$, $\chi'_v = (\omega_f)_v$.
        \item For all $1 \leq \ell \leq \lambda_{d,j^*,k^*}$, $\chi'_{g_\ell}$ is ramified.
    \end{itemize}
    In particular, the cardinality of $\Omega_{\omega_f,g}$ satisfies
    \begin{equation} \label{eqn:prop5.21eq3ac}
        \# \Omega_{\omega_f,g} = \prod_{\ell=1}^{\lambda_{d,j^*,k^*}} \# \Omega_{\omega_f,g_\ell},
    \end{equation}
    where the notations $\Omega_{\omega_f,g_\ell}$ were introduced in Definition \ref{definition:consecutive_local_char}.
    Combining equations (\ref{eqn:prop5.21eq1}), (\ref{eqn:prop5.21eq3}), and (\ref{eqn:prop5.21eq3ab}), we obtain for any given $(g_\ell)_\ell \in \Conf_{\lambda_{d,j^*,k^*}}(\mathcal{P}_{k^*}(d))$ and $J_0 \in \mathbb{Z}_{\geq 0}$,
    \begin{align} \label{eqn:prop5.21eq3b}
    \begin{split}
        & \; \; \; \; \left| (\ref{eqn:prop5.21eq3a}) - \frac{\#\{\omega' \in \Omega_{\omega_f,g} \; | \; \text{rk}(\omega') - \text{rk}(\omega_f) = J_1 - J_0\}}{\# \Omega_{\omega_f,g}} \cdot \# \phi_{d_a}^{-1}(h) \cdot \delta_{h}(J_0) \right| \\
        &< \# \phi^{-1}_{d_a}(h) \cdot 2p \cdot \hat{C}_{E,p,q} \cdot (n \log q)^{-2m_{n,q} + 4 \log p}.
    \end{split}
    \end{align}
    Regardless of the choice of $\ell$, we have $\# \Omega_{\omega_f,g_\ell} = p(p-1)$. Hence, we have
    \begin{align*}
        & \; \; \; \; \sum_{(g_\ell)_\ell \in \Conf_{\lambda_{d,j^*,k^*}}(\mathcal{P}_{k^*}(d))} \frac{\#\{\omega' \in \Omega_{\omega_f,g} \; | \; \text{rk}(\omega') - \text{rk}(\omega_f) = J_1 - J_0\}}{\# \Omega_{\omega_f,g}} \\
        &= \# \Conf_{\lambda_{d,j^*,k^*}}(\mathcal{P}_{k^*}(d)) \cdot \frac{\sum_{(g_\ell)_\ell \in \Conf_{\lambda_{d,j^*,k^*}}(\mathcal{P}_{k^*}(d))} \#\{\omega' \in \Omega_{\omega_f,g} \; | \; \text{rk}(\omega') - \text{rk}(\omega_f) = J_1 - J_0\}}{\sum_{(g_\ell)_\ell \in \Conf_{\lambda_{d,j^*,k^*}}(\mathcal{P}_{k^*}(d))} \# \Omega_{\omega_f,g}}.
    \end{align*}
    Take summation of variants of equations (\ref{eqn:prop5.21eq3b}) over all $(g_\ell)_\ell \in \Conf_{\lambda_{d,j^*,k^*}}(\mathcal{P}_{k^*}(d))$ and use equation (\ref{eqn:prop5.21eq3aa}) to obtain
    \begin{align} \label{eqn:prop5.21eq4}
    \begin{split}
        & \; \; \; \; \left|(\ref{eqn:prop5.21eq3}) - \sum_{(g_\ell)_\ell} \frac{\#\{\omega' \in \Omega_{\omega_f,g} \; | \; \text{rk}(\omega') - \text{rk}(\omega_f) = J_1 - J_0\}}{\# \Omega_{\omega_f,g}} \cdot \# \phi_{d_a}^{-1}(h) \cdot \delta_{h}(J_0)\right| \\
        &= \left|(\ref{eqn:prop5.21eq3}) - \frac{\sum_{(g_\ell)_\ell} \#\{\omega' \in \Omega_{\omega_f,g} \; | \; \text{rk}(\omega') - \text{rk}(\omega_f) = J_1 - J_0\}}{\sum_{(g_\ell)_\ell} \# \Omega_{\omega_f,g}} \cdot \# \phi^{-1}_{d,j^*,k^*}(h) \cdot \delta_{h}(J_0) \right| \\
        &< \# \phi_{d,j^*,k^*}^{-1}(h) \cdot 2p \cdot \hat{C}_{E,p,q} \cdot (n \log q)^{-2m_{n,q} + 4\log p}
    \end{split}
    \end{align}
    where all the summations appearing in the equation above range over $(g_\ell)_\ell \in \Conf_{\lambda_{d,j^*,k^*}}(\mathcal{P}_{k^*}(d))$.

    \medskip
    \textbf{[Determining local statistics]}
    \medskip

    We use the observation that the ranks of the local Selmer groups and the cardinality of $\Omega_{\omega_f,g}$ are invariant with respect to the permutation action of $S_{\lambda_{d,j^*,k^*}}$ on the irreducible divisors of $g$. To avoid confusion, we will use the notation $\widetilde{(g_\ell)}_\ell$ to denote elements in $\PConf_{\lambda_{d,j^*,k^*}}(\mathcal{P}_{k^*}(d))$. Then we obtain the equation
    \begin{align} \label{eqn:prop5.21eq5pre}
        \begin{split}
            & \; \; \; \; \frac{\sum_{(g_\ell)_\ell \in \Conf_{\lambda_{d,j^*,k^*}}(\mathcal{P}_{k^*}(d))} \#\{\omega' \in \Omega_{\omega_f,g} \; | \; \text{rk}(\omega') - \text{rk}(\omega_f) = J_1 - J_0\}}{\sum_{(g_\ell)_\ell \in \Conf_{\lambda_{d,j^*,k^*}}(\mathcal{P}_{k^*}(d))} \# \Omega_{\omega_f,g}} \\
            &= \frac{\sum_{\widetilde{(g_\ell)}_\ell \in \PConf_{\lambda_{d,j^*,k^*}}(\mathcal{P}_{k^*}(d))} \#\{\omega' \in \Omega_{\omega_f,g} \; | \; \text{rk}(\omega') - \text{rk}(\omega_f) = J_1 - J_0\}}{\sum_{\widetilde{(g_\ell)}_\ell \in \PConf_{\lambda_{d,j^*,k^*}}(\mathcal{P}_{k^*}(d))} \# \Omega_{\omega_f,g}}.
        \end{split}
    \end{align}
    Using induction on $\lambda_{d,j^*,k^*}$ and iterating Proposition \ref{prop:Markov_equi} and Proposition \ref{theorem:local_twists} by $\lambda_{d,j^*,k^*}$ many times, we obtain
    \begin{equation} \label{eqn:prop5.21eq5pre2}
        \left|(\ref{eqn:prop5.21eq5pre}) - (M_L^{\tilde{k}} \delta_h)(J_1) \right| < 5 \cdot \lambda_{d,j^*,k^*} \cdot C_{E,p,q} \cdot (n \log q)^{-2m_{n,q} + 6 \log p + 1}.
    \end{equation}
    Because we assume that $d > \mathfrak{n} = \frac{4m_{n,q}^2}{\log q}$ and $w \leq 2 m_{n,q}$, it follows that as long as $m_{n,q} > \deg \Delta_E$, the conditions for applying Proposition \ref{theorem:local_twists} hold. The statement of the proposition follows from combining equations (\ref{eqn:prop5.21eq4}) and (\ref{eqn:prop5.21eq5pre2}). In particular, we obtain
    \begin{align}
    \begin{split}
        & \left| \frac{(\ref{eqn:prop5.21eq3})}{\# \phi_{d,j^*,k^*}^{-1}(h)} - (M_L^{\tilde{k}} \delta_h)(J_1) \right| < \lambda_{d,j^*,k^*} \cdot B_{E,p,q} \cdot ((n \log q)^{-2m_{n,q} + 6 \log p + 1}),
    \end{split}
    \end{align}
    where we can take $B_{E,p,q} = 5 \cdot (2p \cdot \hat{C}_{E,p,q} + C_{E,p,q})$. 
    
    We provide the details of the induction as below. The analogous result for number fields can be found in \cite[Theorem 4.3, Theorem 11.6]{KMR14}.

    \begin{itemize}
        \item \textbf{Base Step}
    \end{itemize}

    Suppose $\lambda_{d,j^*,k^*} = 1$. Then $\PConf_{\lambda_{d,j^*,k^*}}(\mathcal{P}_{k^*}(d)) = \mathcal{P}_{k^*}(d)$ and $g = g_1$. Fix $\omega \in \Omega_{\overline{h}^*}$ such that $\text{rk}(\omega) = J_0$. Proposition \ref{prop:Markov_equi} and Proposition \ref{theorem:local_twists} show that there exists a fixed constant $C_{E,p,q} > 0$ depending only on the elliptic curve $E$, $q$, and $p$ such that
    \begin{align} \label{eqn:prop5.21eq5}
    \begin{split}
        & \; \; \; \; \left| \frac{\sum_{g_1 \in \mathcal{P}_{k^*}(d)} \#\{\omega' \in \Omega_{\omega,g} \; | \; \text{rk}(\omega') - \text{rk}(\omega) = J_1 - J_0\}}{\sum_{g_1 \in \mathcal{P}_{k^*}(d)} \# \Omega_{\omega,g}} - c_{k^*,J_1 - J_0} \right| \\
        &< C_{E,p,q} \cdot (n \log q)^{-2m_{n,q} + 6 \log p + 1}.
    \end{split}
    \end{align}
    The constants $c_{k^*,J_1 - J_0}$ are probabilities obtained from this table, see for example \cite[Proposition 9.5]{KMR14} on how the table from Proposition \ref{theorem:local_twists} is related to the table provided below.
    \begin{center}
    \begin{tabular}{|P{2.5cm}||P{2cm}|P{2cm}|P{5cm}|}
    \hline
    $c_{k^*,J_1 - J_0}$ & $k^* = 0$ & $k^* = 1$ & $k^* = 2$ \\
    \hline
    \hline
    $J_1 - J_0 = -2$ & $\times$ & $\times$ & $1 - (p+1) p^{-J_0} + p^{1-2J_0}$ \\
    \hline
    $J_1 - J_0 = -1$ & $\times$ & $1-p^{-J_0}$ & $\times$ \\
    \hline
    $J_1 - J_0 = 0$ & $1$ & $\times$ & $(p+1)p^{-J_0} - (p + \frac{1}{p})p^{-2J_0}$ \\
    \hline
    $J_1 - J_0 = 1$ & $\times$ & $ p^{-J_0}$ & $\times$ \\
    \hline 
    $J_1 - J_0 = 2$ & $\times$ & $\times$ & $p^{-1-2J_0}$ \\
    \hline
    \end{tabular}
    \end{center}
    It is straightforward to show that the above entries are represented by probabilities obtained from the Markov operator $M_L$ and $M_L^2$. To elaborate,
    \begin{align}
        \begin{split}
            c_{1,-1} &= p_{J_0, J_0-1} \\
            c_{1,1} &= p_{J_0, J_0+1} \\
            c_{2,-2} &= p_{J_0, J_0-1} \cdot p_{J_0-1, J_0-2} \\
            c_{2,0} &= p_{J_0, J_0-1} \cdot p_{J_0-1,J_0} + p_{J_0, J_0+1} \cdot p_{J_0+1, J_0} \\
            c_{2,2} &= p_{J_0, J_0+1} \cdot p_{J_0 + 1, J_0 + 2}.
        \end{split}
    \end{align}
    Summing up $\delta_h(J_0)$ over $5$ possible values of $J_0$ proves the base case for the equation (\ref{eqn:prop5.21eq5pre2}).

    \begin{itemize}
        \item \textbf{Induction step}
    \end{itemize}
    
    Suppose equation (\ref{eqn:prop5.21eq5pre2}) holds up to $\lambda_{d,j^*,k^*} \leq \overline{\lambda}$. As in the base case, fix $\omega \in \Omega_{\overline{h}^*}$ such that $\text{rk}(\omega) = J_0$. Given an element $\widetilde{(g_\ell)}_\ell \in \Conf_{\overline{\lambda}+1}(\mathcal{P}_{k^*}(d))$, we denote by
    \begin{align*}
        \overline{g} := \prod_{\ell=1}^{\overline{\lambda}} g_\ell = \frac{g}{g_{\overline{\lambda}+1}}.
    \end{align*}
    Using Proposition \ref{prop:Markov_equi}, we obtain
    \begin{align*}
        \begin{split}
            & \; \; \; \; \# \{\omega' \in \Omega_{\omega,g} \; | \; \mathrm{rk}(\omega') - \mathrm{rk}(\omega) = J_1 - J_0 \} \\
            &= \sum_{\overline{\omega} \in \Omega_{\omega,\overline{g}}} \# \{\omega' \in \Omega_{\overline{\omega},g_{\overline{\lambda}+1}} \; | \; \mathrm{rk}(\omega') - \mathrm{rk}(\omega) = J_1 - J_0,  \} \\
            &= \sum_{J_2 = - 2\overline{\lambda}}^{2\overline{\lambda}} \left( \sum_{\overline{\omega} \in \Omega_{\omega,\overline{g}}} \# \left\{\omega' \in \Omega_{\overline{\omega},g_{\overline{\lambda}+1}} \; | \; \substack{\mathrm{rk}(\omega') - \mathrm{rk}(\overline{\omega}) = J_1 - \mathrm{rk}(\overline{\omega}) \\ \mathrm{rk}(\overline{\omega}) = J_0 + J_2} \right\} \right).
        \end{split}
    \end{align*}
    This implies the numerator of equation (\ref{eqn:prop5.21eq5pre}) for $\lambda_{d,j^*,k^*} = \overline{\lambda}+1$ can be written as
    \begin{align*}
        \begin{split}
            & \; \; \; \; \sum_{\widetilde{(g_\ell)}_\ell \in \PConf_{\overline{\lambda}+1}(\mathcal{P}_{k^*}(d))} \# \{\omega' \in \Omega_{\omega,g} \; | \; \mathrm{rk}(\omega') - \mathrm{rk}(\omega) = J_1 - J_0 \} \\
            &= \; \; \; \; \sum_{\widetilde{(g_\ell)}_{1 \leq \ell \leq \overline{\lambda}}} \left( \sum_{J_2 = - 2 \overline{\lambda}}^{2 \overline{\lambda}} \left( \sum_{\overline{\omega} \in \Omega_{\omega,\overline{g}}} \left( \sum_{g_{\overline{\lambda}+1}} \# \left\{\omega' \in \Omega_{\overline{\omega},g_{\overline{\lambda}+1}} \; | \; \substack{\mathrm{rk}(\omega') - \mathrm{rk}(\overline{\omega}) = J_1 - \mathrm{rk}(\overline{\omega}) \\ \mathrm{rk}(\overline{\omega}) = J_0 + J_2} \right\} \right) \right) \right),
        \end{split}
    \end{align*}
    where the first summation in the second line of the equation above ranges over $\PConf_{\overline{\lambda}}(\mathcal{P}_{k^*}(d))$, and the last summation in the second line ranges over $\mathcal{P}_{k^*}(d) \setminus \{g_1, \cdots, g_{\overline{\lambda}}\}$. By definition, given a choice of $(g_\ell)_\ell \in \PConf_{\overline{\lambda}+1}(\mathcal{P}_{k^*}(d))$ and $\overline{\omega} \in \Omega_{\omega,\overline{g}}$,
    \begin{equation*}
        \# \Omega_{\omega,g} = \# \Omega_{\overline{\omega},g_{\overline{\lambda}+1}} \cdot \# \Omega_{\omega,\overline{g}} = (p(p-1))^{\overline{\lambda}+1}.
    \end{equation*}
    This implies equation (\ref{eqn:prop5.21eq5pre}) can be rewritten as
    \begin{small}
    \begin{align*}
        \begin{split}
            \frac{1}{\# \PConf_{\overline{\lambda}}(\mathcal{P}_{k^*}(d))}\cdot \sum_{\widetilde{(g_\ell)}_{1 \leq \ell \leq \overline{\lambda}}} \left( \frac{1}{\# \Omega_{\omega,\overline{g}}} \cdot \sum_{J_2 = - 2 \overline{\lambda}}^{2 \overline{\lambda}} \left( \sum_{\overline{\omega} \in \Omega_{\omega,\overline{g}}} \left( \frac{\sum_{g_{\overline{\lambda}+1}} \# \left\{\omega' \in \Omega_{\overline{\omega},g_{\overline{\lambda}+1}} \; | \; \substack{\mathrm{rk}(\omega') - \mathrm{rk}(\overline{\omega}) = J_1 - \mathrm{rk}(\overline{\omega}) \\ \mathrm{rk}(\overline{\omega}) = J_0 + J_2} \right\}}{\sum_{g_{\overline{\lambda}+1}} \# \Omega_{\overline{\omega}, g_{\overline{\lambda}+1}}} \right) \right) \right),
        \end{split}
    \end{align*}
    \end{small}
where as before the summation with entries $\widetilde{(g_\ell)}_{1 \leq \ell \leq \overline{\lambda}}$ ranges over $\PConf_{\overline{\lambda}}(\mathcal{P}_{k^*}(d))$, and the summation with entires $g_{\overline{\lambda}+1}$ ranges over $\mathcal{P}_{k^*}(d) \setminus \{g_1, \cdots, g_{\overline{\lambda}}\}$.
    By Proposition \ref{prop:Markov_equi} and Proposition \ref{theorem:local_twists}, given a fixed choice of $\overline{\omega} \in \Omega_{\omega,\overline{g}}$ such that $\mathrm{rk}(\overline{\omega}) = J_0 + J_2$ for some fixed integers $J_0$ and $-2\overline{\lambda} \leq J_2 \leq 2 \overline{\lambda}$, there exists a fixed constant $C_{E,p,q} > 0$ depending only on the elliptic curve $E,q$, and $p$ such that the innermost terms in the summation satisfy
    \begin{align} \label{eqn:prop5.21eq6}
    \begin{split}
        & \; \; \; \; \left| \frac{\sum_{g_{\overline{\lambda}+1}} \# \left\{\omega' \in \Omega_{\overline{\omega},g_{\overline{\lambda}+1}} \; | \; \substack{\mathrm{rk}(\omega') - \mathrm{rk}(\overline{\omega}) = J_1 - \mathrm{rk}(\overline{\omega}) \\ \mathrm{rk}(\overline{\omega}) = J_0 + J_2} \right\}}{\sum_{g_{\overline{\lambda}+1}} \# \Omega_{\overline{\omega}, g_{\overline{\lambda}+1}}} - c_{k^*, J_1 - (J_0+J_2)} \right| \\
        &< C_{E,p,q} \cdot (n \log q)^{-2 m_{n,q} + 6 \log p + 1}.
    \end{split}
    \end{align}
    The constants $c_{k^*,J_1 - (J_0+J_2)}$ are probabilities obtained from the table below, analogously obtained from the base case where $\lambda_{d,j^*,k^*} = 1$.
    \begin{center}
    \begin{tabular}{|P{3.5cm}||P{1.5cm}|P{2.5cm}|P{6.5cm}|}
    \hline
    $c_{k^*,J_1 - (J_0+J_2)}$ & $k^* = 0$ & $k^* = 1$ & $k^* = 2$ \\
    \hline
    \hline
    $J_1 - (J_0+J_2) = -2$ & $\times$ & $\times$ & $1 - (p+1) p^{-(J_0+J_2)} + p^{1-2(J_0+J_2)}$ \\
    \hline
    $J_1 - (J_0+J_2) = -1$ & $\times$ & $1-p^{-(J_0+J_2)}$ & $\times$ \\
    \hline
    $J_1 - (J_0+J_2) = 0$ & $1$ & $\times$ & $(p+1)p^{-(J_0+J_2)} - (p + \frac{1}{p})p^{-2(J_0+J_2)}$ \\
    \hline
    $J_1 - (J_0+J_2) = 1$ & $\times$ & $ p^{-(J_0+J_2)}$ & $\times$ \\
    \hline 
    $J_1 - (J_0+J_2) = 2$ & $\times$ & $\times$ & $p^{-1-2(J_0+J_2)}$ \\
    \hline
    \end{tabular}
    \end{center}
    And analogous to the base case, the above entries are represented by probabilities obtained from the Markov operator $M_L$ and $M_L^2$. 
    
    Consider the expression
    \begin{small}
    \begin{equation} \label{eqn:prop5.21eq6a}
        \frac{1}{\# \PConf_{\overline{\lambda}}(\mathcal{P}_{k^*}(d))} \cdot \sum_{\widetilde{(g_\ell)}_\ell} \left( \frac{\sum_{J_2 = -2\overline{\lambda}}^{2\overline{\lambda}} \#\{\overline{\omega} \in \Omega_{\omega,\overline{g}} \; | \; \mathrm{rk}(\overline{\omega}) = J_0 + J_2\} \cdot c_{k^*,J_1-(J_0+J_2)}}{\# \Omega_{\omega,\overline{g}}} \right),
    \end{equation}
    \end{small}
    where the summation $\widetilde{(g_\ell)}_\ell$ ranges over $\PConf_{\overline{\lambda}}(\mathcal{P}_{k^*}(d))$.
    Equation (\ref{eqn:prop5.21eq6}) implies
    \begin{align} \label{eqn:prop5.21eq7a}
        |(\ref{eqn:prop5.21eq5pre}) - (\ref{eqn:prop5.21eq6a})| < 5 \cdot C_{E,p,q} \cdot (n \log q)^{-2 m_{n,q} + 6 \log p + 1}.
    \end{align}
    The induction hypothesis for equation (\ref{eqn:prop5.21eq5pre2}) implies
    \begin{align} \label{eqn:prop5.21eq7b}
        |(\ref{eqn:prop5.21eq6a}) - (M_L^{k^*\cdot(\overline{\lambda}+1)}\delta_h)(J_1)| < 5 \cdot \overline{\lambda} \cdot C_{E,p,q} \cdot (n \log q)^{-2 m_{n,q} + 6 \log p + 1}.
    \end{align}
    Combining equations (\ref{eqn:prop5.21eq7a}) and (\ref{eqn:prop5.21eq7b}) gives equation (\ref{eqn:prop5.21eq5pre2}) for $\lambda_{d,j^*,k^*} = \overline{\lambda}+1$.
\end{proof}
\begin{remark}
    One may regard Proposition \ref{prop:local_twist_2} as an effective version of \cite[Theorem 4.3, Theorem 9.5]{KMR14}. Instead of using fan structures, we consider a subset of polynomials over $\phi_{d,j^*,k^*}^{-1}(h)$ to show that the Markov chain $M_L$ governs the probability distribution of ranks of local Selmer groups with explicitly computable rate of convergence. 
\end{remark}

\section{Global Selmer groups} \label{section:global_selmer}

The goal of this section is to use the probability distribution of $\text{rk}((\chi_{f,v})_v)$ ranging over $F_{(n,N),(w,w')}^{(\lambda,\eta)}(\F_q)$ (Proposition \ref{prop:local_twist_2}) to prove the statement of the main theorem.

\subsection{Governing Markov operator} \label{section:governing_markov_chain}

We will use the Markov operator constructed from \cite{KMR14}, known as the mod $p$ Lagrangian operator, to analyze variations of $\pi$-Selmer ranks of a subfamily of global quadratic twists of elliptic curves over $K$ satisfying the conditions from Theorem \ref{theorem:main_theorem}. 

\begin{definition} \label{defn:markov_chain_KMR}
Let $M_L = [p_{r,s}]$ be the operator over the state space of non-negative integers $\mathbb{Z}_{\geq 0}$ given by
\begin{equation*}
    p_{r,s} = \begin{cases}
    1 - p^{-r} &\text{ if } s = r-1 \geq 0, \\
    p^{-r} &\text{ if } s = r+1, \\
    0 &\text{ else}.
    \end{cases}
\end{equation*}
\end{definition}

\begin{remark}
The construction of the mod $p$ Lagrangian Markov operator dates back to previous works by \cite{SD08} and \cite{KMR14}. Other references such as \cite{Sm17}, \cite{Sm20}, and \cite{FLR20} also use Markov chains to obtain the probability distribution of $p$-Selmer groups of certain families of elliptic curves.
\end{remark}

We list some crucial properties the operator $M_L$ satisfies, the proof of which can be found in \cite[Section 2]{KMR14}.

\begin{definition}
Let $\mu: \mathbb{Z}_{\geq 0} \to [0,1]$ be a probability distribution over the state space of non-negative integers $\mathbb{Z}_{\geq 0}$. The parity of $\mu$ is the sum of probabilities at odd state spaces, i.e.
\begin{equation*}
    \rho(\mu) := \sum_{n \text{ odd}} \mu(n). 
\end{equation*}
\end{definition}

\begin{proposition} \cite[Proposition 2.4]{KMR14} \label{proposition:property_markov_chain}

Let $E^+, E^-: \mathbb{Z}_{\geq 0} \to [0,1]$ be probability distributions such that
\begin{equation*}
    E^+(n) = \begin{cases}
    \prod_{j=1}^\infty (1 + p^{-j})^{-1} \prod_{j=1}^n \frac{p}{p^j - 1} &\text{ if n even}, \\
    0 &\text{ if n odd}.
    \end{cases}
\end{equation*}
\begin{equation*}
    E^-(n) = \begin{cases}
    0 &\text{ if n even}, \\
    \prod_{j=1}^\infty (1 + p^{-j})^{-1} \prod_{j=1}^n \frac{p}{p^j - 1} &\text{ if n odd}.
    \end{cases}
\end{equation*}
Let $\mu: \mathbb{Z}_{\geq 0} \to [0,1]$ be a probability distribution. Then 
\begin{align*}
    \lim_{k \to \infty} M_L^{2k}(\mu) &= (1 - \rho(\mu)) E^+ + \rho(\mu) E^-, \\
    \lim_{k \to \infty} M_L^{2k+1}(\mu) &= \rho(\mu)E^+ + (1-\rho(\mu)) E^-.
\end{align*}
In particular, if $\rho(\mu) = \frac{1}{2}$, then
\begin{equation}\label{equation:stationary_distribution}
    \lim_{k \to \infty} M_L^{k}(\mu)(n) = \prod_{j \geq 0}^\infty (1 + p^{-j})^{-1} \prod_{j=1}^n \frac{p}{p^j - 1}.
\end{equation}
\end{proposition}

\begin{remark}
Note that $M_L^2$ is an aperiodic, irreducible, and positive-recurrent Markov chain over the state space of positive odd integers $\Z_{odd, \geq 0}$ and non-negative even integers $\Z_{even, \geq 0}$. The unique stationary distributions of the Markov chain are $E^-(n)$ and $E^+(n)$, respectively.
\end{remark}

Given that $M_L^2$ is aperiodic, irreducible, and positive-recurrent, it is natural to ask what the rate of convergence of $M_L$ is. Assuming certain conditions on the initial probability distribution over the state space and the stationary distribution of $M$, the geometric rate of convergence of $M$ can be verified using the following theorem.

\begin{theorem}[Geometric ergodic theorem for Markov chains] \cite[Theorem 15.0.1]{MT93} \label{theorem:geometric_markov}

Let $M$ be an irreducible, aperiodic, and positive-recurrent Markov chain over a countable state space $\mathcal{X} := (x_n)_{n \in \Z}$. Let $X_1, X_2, \cdots, X_n, \cdots : \mathcal{X} \to [0,1]$ be a sequence of random variables which satisfy
\begin{equation}
    X_{n+1} = M(X_{n})
\end{equation}
for all $n$. Let $\pi$ be an invariant probability distribution of $M$ (not necessarily unique). Let $V: X \to [1, \infty)$ be a function such that $\lim_{n \to \infty} V(x_n) = \infty$. 
Denote by $\mathbb{E}[V(\mu)]$ the expected value of the probability distribution $V(\mu): [1,\infty) \to [0,1]$, i.e.
\begin{equation*}
    \mathbb{E}[V(\mu)] := \sum_{n \in \mathbb{Z}} V(x_n) \cdot \mu(x_n).
\end{equation*}
Given a state $x \in \mathcal{X}$, we denote by $\mu_x$ the probability distribution defined as 
\begin{equation*}
    \mu_x(z) = \begin{cases}
        1 &\text{ if } z = x \\
        0 &\text{ otherwise}.
    \end{cases}
\end{equation*}
If there exists $0 < \rho < 1$ and a fixed $\kappa < \infty$ such that,
\begin{equation} \label{eqn:drift}
    \mathbb{E}[V(M(\mu_{x}))] - V(x) \leq \begin{cases}
        \kappa &\text{ for finitely many } x \in \mathcal{X}, \\
        -\rho V(x) &\text{ otherwise },
    \end{cases}
\end{equation}
then there exists a constant $0 <\gamma < 1$ and a constant $c>0$ such that for any probability distribution $\mu$ over $X$ and every $n \in \mathbb{N}$,
\begin{equation} \label{eqn:geom-erg}
    \sup_{\substack{z \in X}} | M^n(\mu)(z) - \pi | < c \gamma^n (\mathbb{E}[V(\mu)] + 1),
\end{equation}
where the term $\mathbb{E}[V(\mu)]$ is the expected value of $V$ under the probability distribution $\mu$.
\end{theorem}
We would like to thank the anonymous reviewer for pointing out this important observation. The theorem establishes a relation between geometric ergodicity (equation (\ref{eqn:geom-erg})) and drift inequality (equation (\ref{eqn:drift})) associated to Markov chains. The relation, however, is ineffective in a sense that the statement does not imply any relation between the rates $\gamma$ and $\rho$.

Let $I$ be the identity operator over the countable state space $\mathbb{Z}_{\geq 0}$. Proposition \ref{theorem:local_twists} implies that the Markov chain 
\begin{equation} \label{equation:markov_chain_interest}
    \left( 1 - \frac{p}{(p^2-1)} \right) \cdot I + \frac{1}{p} M_L + \frac{1}{(p^3-p)} M_L^2
\end{equation}
over the state space $\Z_{\geq 0}$ governs the differences between the dimensions of two local Selmer groups $\Sel(E[p],\chi')$ and $\Sel(E[p],\chi)$ where $\chi' \in \Omega_{\chi, \mathfrak{v}}$ for some place $\mathfrak{v}$, i.e. except at the place $\mathfrak{v}$, the Cartesian product of local characters $\chi'$ is identical to $\chi$. Proposition \ref{proposition:property_markov_chain} also shows that regardless of the parity of the initial probability distribution over the state space $\Z_{\geq 0}$, the stationary distribution of the Markov chain from (\ref{equation:markov_chain_interest}) is given by the Poonen-Rains distribution as stated in (\ref{equation:stationary_distribution}). One can also show that given a fixed prime number $p$, the Markov chain of our interest is an irreducible aperiodic Markov chain over the countably infinite state space $\Z_{\geq 0}$. In fact, it is geometrically ergodic over $\Z_{\geq 0}$ (without requiring the restriction that $p = 2$).

\begin{corollary}\label{corollary:uniform_markov}
Let $\mu: \Z_{\geq 0} \to [0,1]$ be a probability distribution over the state space $\Z_{\geq 0}$. Denote by $\pi$ the stationary probability distribution of the Markov operator given by
\begin{equation}
    M := \left( 1 - \frac{p}{(p^2-1)} \right) \cdot I + \frac{1}{p} M_L + \frac{1}{(p^3-p)} M_L^2.
\end{equation}
for some fixed prime number $p$ and a finite cyclic group $T$. Then for every $n \in \mathbb{N}$, there exists a constant $0 \leq \gamma_p < 1$ depending on $p$ and a constant $c > 0$ such that 
\begin{equation}
    \sup_{z \in X} \left| \left(\left( 1 - \frac{p}{(p^2-1)} \right) \cdot I + \frac{1}{p} M_L + \frac{1}{(p^3-p)} M_L^2 \right)^n(\mu) - \pi \right| < c \gamma_p^n (\mathbb{E}[p^\mu] + 1).
\end{equation}
where the term $\mathbb{E}[p^\mu]$ is the expected value $\mathbb{E}[V(\mu)]$ with $V(x) = p^x$.
\end{corollary}
\begin{proof}
Set $V(x) = p^x$. Recall that given any $x \in \mathcal{X}$, we denote by $\mu_x$ the probability distribution that achieves probability $1$ at state $x$ and $0$ elsewhere. Computational results then show that there exists a fixed constant $\kappa < \infty$ such that for every $x \in \mathbb{Z}_{\geq 0}$,
\begin{align*}
    \mathbb{E} \left[ p^{M(\mu_x)}\right] &= \left( 1 - \frac{p^2 - p + 1}{p^3} \right) \cdot p^x + \left( 1 + \frac{1}{p^3} \right).
\end{align*}
The corollary follows from Theorem \ref{theorem:geometric_markov} by setting $\rho = - \frac{p^4-2\cdot p^3+p^2-1}{p^5}$, $\kappa = p + \frac{1}{p} - \frac{1}{p^2} + \frac{1}{p^3}$, and the finite set of states of $\mathbb{Z}_{\geq 0}$ to be $\{0,1\}$.
\end{proof}

While Theorem \ref{theorem:geometric_markov} does not establish effective relations between $\gamma$ and $\rho$, one can still obtain the desired effective relations for Markov chains satisfying certain conditions, see for example \cite{Sp92}, \cite{MT94}, \cite{Bax05}, and \cite{GHLR24}. For the Markov chain $M$ in equation (\ref{equation:markov_chain_interest}), the work by Baxendale \cite{Bax05} can be used to obtain unconditional numerical approximations of non-optimal lower bounds for $\gamma_p$. Suppose a Markov chain $M$ satisfying the drift condition (equation (\ref{eqn:drift})) from Theorem \ref{theorem:geometric_markov} over a countable state space $\mathcal{X}$ also satisfies the following condition (termed as ``Minorization condition'' in \cite[Section 1]{Bax05}): There exists a finite set $C \subset \mathcal{X}$, a probability measure $\nu: \mathcal{X} \to [0,1]$ such that $\nu(C) = 1$, and $\beta > 0$ such that
\begin{equation*}
    \sum_{z \in A} (M(\mu_x))(z) \geq \beta \cdot \nu(A)
\end{equation*}
for all $x \in C$ and all subsets $A \subset \mathcal{X}$. For the Markov chain $M$ in equation (\ref{equation:markov_chain_interest}), we can take $\mathcal{X} = \mathbb{Z}_{\geq 0}$ and  the parameters $C, \beta, \nu$ as follows:
\begin{align}
    \begin{split}
        C &:= \{0,1\}, \; \; \; \; \beta = \begin{cases} \frac{23}{32} &\text{ if } p=2, \\ \frac{2p-1}{p^2} &\text{ if } p \geq 3, \end{cases}, \; \; \; \; \nu(z) = \begin{cases} \frac{p-1}{2p-1} &\text{ if } z=0, \\ \frac{p}{2p-1} &\text{ if } z=1, \\ 0 &\text{ otherwise}.\end{cases}
    \end{split}
\end{align}
We note that the choices of $C,\beta,\nu$ above do not necessarily give the optimal value for $\gamma_p$. Define the following constants appearing in \cite[Section 2]{Bax05}:
\begin{align}
    \begin{split}
        \alpha_1 &:= 1 - \frac{\log(\kappa - \beta) - \log(1 - \beta)}{\log(\rho)}, \; \; \; \; \alpha_2 := 1, \; \; \; \; R_0 := \min(1/\rho, (1-\beta)^{-1/\alpha_1}).
    \end{split}
\end{align}
Note that we can take $\alpha_2 = 1$ because $\nu(C) = 1$. By simplifying the expression appearing in \cite[equation (4)]{Bax05} and using the fact that $C$ is a non-atomic set, the geometric rate of convergence $\gamma_p$ of the Markov chain $M$ satisfies
\begin{equation}
   2 \cdot \min_{R \in [1,R_0]} \left( \left(1 + \sqrt{1 + \frac{e^2 \cdot \beta \cdot (R-1) \cdot (1-(1-\beta) \cdot R^{\alpha_1}) \cdot(\log R)^2}{2 \cdot (\beta \cdot R - 1 + (1-\beta) \cdot R^{\alpha_1})}}\right)^{-1} \right) < \gamma_p < 1.
\end{equation}
Provided below is the numerical approximation of non-optimal admissible values of geometric rate of convergence $\gamma_p$ for primes $p = 2,3,5,7$, whose lower bounds are approximated up to 10 digits.
\begin{itemize}
    \item $p=2$: $0.9996768309 < \gamma_2 < 1$.
    \item $p=3$: $0.9998797848 < \gamma_3 < 1$.
    \item $p=5$: $0.9999942992 < \gamma_5 < 1$.
    \item $p=7$: $0.9999994169 < \gamma_7 < 1$.
\end{itemize}

It now remains to show that the stationary distribution of the desired Markov chain (\ref{equation:markov_chain_interest}) is the probability distribution conjectured by Poonen-Rains \cite{BKLPR15}. 
\begin{lemma} \label{lemma:markov_chain_interest_stat}
Let $p$ be any fixed prime number. Then the probability distribution
\begin{equation}
    PR(j) := \prod_{j \geq 0}^\infty (1 + p^{-j})^{-1} \prod_{j=1}^n \frac{p}{p^j - 1}
\end{equation}
is the unique stationary distribution of the Markov chain 
\begin{equation}
    M := \left( 1 - \frac{p}{(p^2-1)} \right) \cdot I + \frac{1}{p} M_L + \frac{1}{(p^3-p)} M_L^2.
\end{equation}
\end{lemma}
\begin{proof}
Note that the operators $I$ and $M_L^2$ are parity preserving Markov operators, whereas $M_L$ is a parity reversing Markov operator. Because $M$ is aperiodic and irreducible, it follows that $M$ has a unique stationary distribution $\pi$. The following relation holds for the parity of $\pi$, which is obtainable by comparing the parity between $\pi$ and $M(\pi)$.
\begin{equation}
    \rho(\pi) = \left(1 - \frac{1}{p} \right) \rho(\pi) + \frac{1}{p} \left(1 - \rho(\pi) \right) = \left(1 - \frac{2}{p} \right)\rho(\pi) + \frac{1}{p}.
\end{equation}
Therefore, we obtain that $\rho(\pi) = \frac{1}{2}$. Using Proposition \ref{proposition:property_markov_chain} and the fact that the Markov chain $M$ is aperiodic and irreducible, we immediately obtain the statement of the lemma.
\end{proof}

\begin{remark}
One crucial result from using Corollary \ref{corollary:uniform_markov} and Lemma \ref{lemma:markov_chain_interest_stat} is that the stationary distribution of applying the Markov chain from (\ref{equation:markov_chain_interest}) is equal to the Poonen-Rains distribution regardless of the initial probability distribution. Furthermore, as long as the initial probability distribution is finitely supported, we can also ensure that the Markov chain converges to the stationary distribution at a geometric convergence rate.
\end{remark}

\begin{remark}
We note that the Markov chain constructed from Smith's work is different from the Markov chain presented in this manuscript \cite{Sm22_01, Sm22_02}. The sequence of random variables $X_n$ Smith considers correspond to the empirical probability distribution of the subspace
\begin{equation}
    \dim_{\F_p} \pi^{n-1} \Sel_{\pi^n}(E^\chi) \subset \Sel_{\pi} (E)
\end{equation}
where $\chi$ ranges over grids of twists \cite[Chapter 6]{Sm22_01}. Here, the grids of twists are defined as a finite Cartesian product of collections of prime ideals, where each collection contains prime ideals whose symbols are equal to each other \cite[Definition 4.13]{Sm22_01}. 

To elaborate, this manuscript regards the variable $n$ from a sequence of random variables $\{X_n\}_{n \in \mathbb{Z}}$ as the number of distinct irreducible places, whereas Smith's work regards the variable $n$ from a sequence of random variables $\{X_n\}_{n \in \mathbb{Z}}$ as a quantifier for detecting elements inside higher $\pi^n$-Selmer groups which also lie inside the $\pi$-Selmer group of $E$.
\end{remark}

\subsection{Relating global and local Selmer groups} \label{section:proofmain}

We now obtain the desired probability distribution of dimensions of $\Sel_\pi(E^{\chi_f})$ over $f \in F_n(\F_q)$ by approximating it with distribution of dimensions of local Selmer groups of $E$ associated to restrictions of $\chi_f$, as stated in Proposition \ref{prop:local_twist_2}.

\begin{proposition} \label{prop:global-local}
    Let $n > N$ and $w < 2m_{n,q}$ be positive integers. Let $w'$ be a positive integer such that $w' = (1-\epsilon)w$ for some small enough $0 < \epsilon < 1$.
    
    Suppose that $n$ satisfies the following inequality
    \begin{equation}
        m_{n,q} > \text{max} \left(e^{e^e}, \deg \Delta_E, 6\log p + 2 \right).
    \end{equation}
    Then there exists a fixed constant $\tilde{B}_{E,p,q}$ depending only on $E, p, q$ such that
    \begin{align}
        \begin{split}
            & \; \; \; \; \left| \frac{\#\{f \in \hat{F}_{(n,N),(w,w')}(\F_q) \; | \; \dim_{\mathbb{F}_p} \Sel_\pi(E^{\chi_f}) = J\}}{\# \hat{F}_{(n,N),(w,w')}(\F_q)} - PR(J) \right| \\
        &< \tilde{B}_{E,p,q} \cdot (n\log q)^{4 \epsilon \log p} \cdot \left( (n\log q)^{-m_{n,q}} + \gamma_p^{w'-1} \right).
        \end{split}
    \end{align}
    where $\hat{F}_{(n,N),(w,w')}(\F_q)$ is a subset of $F_n(\F_q)$ as stated in Definition \ref{defn:polynomial_sets}, and $\gamma_p$ is the geometric rate of convergence of the Markov operator $M$ as stated in Corollary \ref{corollary:uniform_markov}.
\end{proposition}
As stated in previous sections, the error term appearing in Proposition \ref{prop:global-local} corresponds to one of the error terms constituting the constant $\alpha(p)$ defined in Theorem \ref{theorem:main_theorem}.
\begin{proof}

\medskip 
\textbf{[[Setup]]}
\medskip 

Before presenting the proof of the proposition, we first outline the set of notations utilized in the proof. We recall that there exists a $\text{Gal}(\overline{K}/K)$-equivariant isomorphism
\begin{equation}
    E^{\chi_f}[\pi] \cong E[p],
\end{equation}
see \cite[Proposition 4.1]{MR07} for the proof. This implies that the $\pi$-Selmer group of $E^{\chi_f}$ satisfies
\begin{equation}
    \Sel_\pi(E^{\chi_f}) \subset H^1_{\et}(K, E[p]),
\end{equation}
and the image of the local Kummer maps $\text{im} \delta_v^{\chi}$ are Lagrangian subspaces of $H^1_{\et}(K_v,E[p])$ for each place $v$ of $K$. The $\pi$-Selmer group of $E^{\chi_f}$ is hence the local Selmer group of $E$ associated to the Cartesian product $(\chi_{f,v})_v$ arising from restrictions of the global character $\chi_f$ to cyclic order-$p$ local characters over some local fields $K_v$. We concretely have
\begin{equation}
    \Sel_\pi(E^{\chi_f}) = \Sel(E[p], (\chi_{f,v})_{v \in \Sigma_f(\overline{f}^*)}) \in \Omega_{\overline{f}^*}.
\end{equation}
The relation between $\pi$-Selmer groups and local Selmer groups also holds over number fields as well, see for example \cite[Chapter 10]{KMR14}.

For each positive integer $1 \leq z \leq w'$, let
    \begin{equation}
        \mathfrak{d}_z := \min\{ d > \mathfrak{n} \; | \; \sum_{i=\mathfrak{n}+1}^d \sum_{j=1}^{p-1} \sum_{k=0}^2 \lambda_{i,j,k} < z\}.
    \end{equation}
    In other words, it is the $z$-th lowest degree of distinct irreducible factors of $f^*$. We define polynomials $f_{\mathfrak{d}_z}$ as follows:
    \begin{equation}
        f_{\mathfrak{d}_z} := \prod_{\substack{g \mid f^* \\ g \in \cup_{i=\mathfrak{n}+1}^{\mathfrak{d}_z}  \mathcal{P}_1(i) \cup \mathcal{P}_2(i)}} g^{v_g(f)},
    \end{equation}
    i.e. it is the product of irreducible factors of $f \in F_{(n,N),(w,w')}^{(\lambda,\eta)}(\F_q)$ (including multiplicities) up to $z$-th lowest degree exceeding $\mathfrak{n}$ that do not lie in $\mathcal{P}_0$. We now define the following abbreviation of local characters for each $1 \leq z \leq w'$:
    \begin{align}
    \begin{split}
        \chi_{f,0} &:= (\chi_{f,v})_{v \in \Sigma_f}, \; \; \; \; \chi_{f,z} := (\chi_{f,v})_{v \in \Sigma_f \cup (\overline{f_{\mathfrak{d}_z}})}.
    \end{split}
    \end{align}
    In other words, $\chi_{f,z}$ is the Cartesian product of restriction of the global character $\chi_f$ over places in $\Sigma_f$ and places of degree at most the $z$-th lowest degree of distinct irreducible factors of $f^*$. Using these notations, we have
    \begin{equation}
        \Sel_\pi(E^{\chi_f}) = \Sel(E[p], \chi_{f,w'}).
    \end{equation}
    
    Let $\lambda \in \Lambda_{N,w'}^{la}$ and $\eta \in \Lambda_{n-N,w-w'}^{for}$.
    There is a projection map which forgets all irreducible factors of degree greater than $\mathfrak{n}$:
    \begin{align*}
    \begin{split}
        \Phi: F_{(n,N),(w,w')}^{(\lambda,\eta)}(\F_q) = \left[ \prod_{i,j,k} \Conf_{\lambda_{i,j,k}}(\mathcal{P}_k(i)) \right] \times \left[ \prod_{\hat{i},\hat{j},\hat{k}} \Conf_{\eta_{\hat{i},\hat{j},\hat{k}}}(\mathcal{P}_{\hat{k}}(\hat{i})) \right] \to \left[ \prod_{\hat{i},\hat{j},\hat{k}} \Conf_{\eta_{\hat{i},\hat{j},\hat{k}}}(\mathcal{P}_{\hat{k}}(\hat{i})) \right].
    \end{split}
    \end{align*}

    \medskip
    \textbf{[[Statistics over fibers of $\Phi$]]}
    \medskip
    
    Suppose that $h_* \in F_{n-N}(\F_q)$ such that $h_*$ admits the forgetful partition $\eta$. Given such a choice of $h_*$, we will pay particular focus to the set of fibers $\Phi^{-1}(h_*)$.
    We then have:
    \begin{align} \label{eqn:maineq1}
    \begin{split}
        & \; \; \; \; \# \{f \in \Phi^{-1}(h_*) \; | \; \dim_{\mathbb{F}_p} \Sel_\pi(E^{\chi_f}) = J\} \\
        &= \# \{f \in \Phi^{-1}(h_*) \; | \; \text{rk}(\chi_{f,w'}) = J\} \\
        &= \sum_{J_0 = 0}^{\infty} \# \left\{f \in \Phi^{-1}(h_*) \; | \; \text{rk}(\chi_{f,0}) = J_0, \sum_{z=1}^{w'} \text{rk}(\chi_{f,z}) - \text{rk}(\chi_{f,z-1}) = J\right\}.
    \end{split}
    \end{align}

    Denote by $\Omega_{\overline{h_*}}$ the following set of Cartesian product of local characters
    \begin{equation}
        \Omega_{\overline{h_*}} := \prod_{v \in \Sigma_E} \text{Hom}(\text{Gal}(\overline{K}_v/K_v), \mu_p) \times \prod_{v \mid h_*} \text{Hom}(\text{Gal}(\overline{K}_v/K_v), \mu_p) \subset \Omega_1.
    \end{equation}
    Let $\delta_{h_*}: \mathbb{Z}_{\geq 0} \to [0,1]$ be the probability distribution defined as
    \begin{equation}
        \delta_{h_*}(J) := \frac{\# \{\omega \in \Omega_{\overline{h_*}} \; | \; \text{rk}(\omega) = J\}}{\# \Omega_{\overline{h_*}}}.
    \end{equation}
    Let $d_\lambda$ be an integer associated to a choice of a splitting partition $\lambda$ defined as
    \begin{equation}
        d_\lambda := \sum_{i,j} (\lambda_{i,j,1} + 2 \cdot \lambda_{i,j,2}).
    \end{equation}
    Note that there exists a bijection
    \begin{equation*}
        \Phi^{-1}(h) \cong \prod_{i,j,k} \Conf_{\lambda_{i,j,k}}(\mathcal{P}_k(i)).
    \end{equation*}
    Inductively applying Proposition \ref{prop:local_twist_2} to each term $\Conf_{\lambda_{i,j,k}}(\mathcal{P}_k(i))$, we obtain that
    \begin{align} \label{eqn:maineq3}
        \begin{split}
            & \; \; \; \; \left| \frac{\#\{f \in \Phi^{-1}(h_*) \; | \; \dim_{\mathbb{F}_p} \Sel_\pi(E^{\chi_f}) = J\}}{\# \Phi^{-1}(h_*)} - (M_L^{d_\lambda} \delta_{h_*})(J) \right| \\
            &< B_{E,p,q} \cdot d_\lambda \cdot (n \log q)^{-2m_{n,q} + 6 \log p + 1} < B_{E,p,q} \cdot (n \log q)^{-2m_{n,q} + 6 \log p + 2},
        \end{split}
    \end{align}
    where $B_{E,p,q} > 0$ is the explicit constant constructed in Proposition \ref{prop:local_twist_2}.

    \medskip
    \textbf{[[Statistics over unions of fibers of $\Phi$]]}
    \medskip

    Denote by ${F}_{(n,N),(w,w')}^{h_*}(\F_q)$ the disjoint union of subsets
    \begin{align}
    \begin{split}
        {F}_{(n,N),(w,w')}^{h_*}(\F_q) &:= \bigsqcup_{\lambda \in \Lambda_{N,w'}^{la}} \Phi^{-1}(h_*).
    \end{split}
    \end{align}
    Recall that we defined the Markov operator $M$ over $\mathbb{Z}_{\geq 0}$ as
    \begin{equation}
        M := \left(1 - \frac{p}{p^2-1}\right) \cdot I + \frac{1}{p} M_L + \frac{1}{p^3 - p} M_L^2.
    \end{equation}
    Summing variants of equation (\ref{eqn:maineq3}) over the set of locally arrangeable partitions $\Lambda_{N,w'}^{la}$, we obtain
    \begin{align}
        \begin{split}
            & \; \; \; \; \left| \frac{\#\{f \in {F}_{(n,N),(w,w')}^{h_*}(\F_q) \; | \; \dim_{\mathbb{F}_p} \Sel_\pi(E^{\chi_f}) = J\}}{ \# {F}_{(n,N),(w,w')}^{h_*}(\F_q)} - (M^{w'-1} \delta_{h_*})(J) \right| \\
            &< B_{E,p,q} \cdot (n \log q)^{-2m_{n,q} + 6 \log p + 2} \\
            &< B_{E,p,q} \cdot (n \log q)^{-m_{n,q}}.
        \end{split}
    \end{align}
    Note that we iterate the Markov chain $M$ by $w'-1$ times, rather than $w'$ times, because we are using one of the auxiliary places of $f$ to obtain an equidistribution of characters $\{\chi_{f,w'}\}$ inside $\Omega_{\Sigma_f(\overline{f}^*)}$, hence allowing us to apply Proposition \ref{prop:local_twist_2}. 

    \medskip
    \textbf{[[Incorporating ergodicity of Markov chains]]}
    \medskip

    Recall the Poonen-Rains distribution
    \begin{equation*}
        PR(J) = \prod_{j \geq 0}^\infty \frac{1}{1 + p^{-j}} \prod_{j=1}^J \frac{p}{p^j - 1}.
    \end{equation*}
    Because we set $w - w' = \epsilon w$ for small enough $0 < \epsilon < 1$, it follows that
    \begin{equation}
        \max_{J \in \mathbb{Z}_{\geq 0}} \{J \; | \; \delta_{h_*}(J) \neq 0 \} \leq \max_{\chi \in \Omega_E} \text{rk}(\chi) + 2 \epsilon w.
    \end{equation}  
    By Corollary \ref{corollary:uniform_markov}, we obtain that there exists a fixed constant $c > 0$ such that
    \begin{equation} \label{eqn:maineq2}
        \sup_{J \in \mathbb{Z}_{\geq 0}} \left| (M^{w'-1} \delta_{h_*})(J) - PR(J) \right| < c \cdot \gamma_p^{w'-1} \cdot \mathbb{E}[p^{\delta_{h^*}}],
    \end{equation}
    where we recall that $\gamma_p$ is the geometric rate of convergence of the Markov operator $M$ as stated in Corollary \ref{corollary:uniform_markov}.
    Because $w \leq 2 m_{n,q}$, it follows that
    \begin{equation}
        \mathbb{E}[p^{\delta_{h_*}}] \leq p^{\max_{\chi \in \Omega_E} \text{rk}(\chi)}  \cdot (n \log q)^{4 \epsilon \log p}.
    \end{equation}
    By letting $c_p := c \cdot p^{\max_{\chi \in \Omega_E} \text{rk}(\chi)}$, we obtain:
    \begin{equation}
        (\ref{eqn:maineq2}) < c_p \cdot \gamma_p^{w'-1} \cdot (n \log q)^{4 \epsilon \log p}.
    \end{equation} 
    Using triangle inequality with equation (\ref{eqn:maineq3}), we obtain for all $J \geq 0$ and for any small enough $0 < \epsilon < 1$, there exists an explicit constant $\tilde{B}_{E,p,q} := B_{E,p,q} + c_p$ such that
    \begin{align}
    \begin{split}
        & \; \; \; \; \left|\frac{\#\{f \in {F}_{(n,N),(w,w')}^{h_*}(\F_q) \; | \; \dim_{\mathbb{F}_p} \Sel_\pi(E^{\chi_f}) = J\}}{\# {F}_{(n,N),(w,w')}^{h_*}(\F_q)} - PR(J) \right| \\
        &< \tilde{B}_{E,p,q} \cdot (n\log q)^{4 \epsilon \log p} \cdot \left( (n\log q)^{-m_{n,q}} + \gamma_p^{w'-1} \right).
    \end{split}
    \end{align}

    \medskip
    \textbf{[[Statistics over $\hat{F}_{(n,N),(w,w')}(\F_q)$]]}
    \medskip
    
    Denote by $F_{(n,N),(w,w')}^{\eta}(\F_q)$ the following disjoint union of subsets
    \begin{align}
    \begin{split}
        F_{(n,N),(w,w')}^{\eta}(\F_q) &:= \bigsqcup_{\substack{h_* \in F_{n-N}(\F_q) \\ h_* \text{ admits } \eta}} F_{(n,N),(w,w')}^{h_*}(\F_q).
    \end{split}
    \end{align}    
    By ranging over all $h_* \in F_{n-N}(\F_q)$ such that $h_*$ admits the forgettable splitting partition $\eta$, we obtain that
    \begin{align}
    \begin{split}
        & \; \; \; \; \left| \frac{\#\{f \in F_{(n,N),(w,w')}^{\eta}(\F_q) \; | \; \dim_{\mathbb{F}_p} \Sel_\pi(E^{\chi_f}) = J\}}{\# F_{(n,N),(w,w')}^{\eta}(\F_q)} - PR(J) \right| \\
        &< \tilde{B}_{E,p,q} \cdot (n\log q)^{4 \epsilon \log p} \cdot \left( (n\log q)^{-m_{n,q}} + \gamma_p^{w'-1} \right).
    \end{split}
    \end{align}
    Recall that $\hat{F}_{(n,N),(w,w')}(\F_q)$ is the following disjoint union of sets:
    \begin{equation}
        \hat{F}_{(n,N),(w,w')}(\F_q) := \bigsqcup_{\lambda \in \Lambda_{N,w'}^{la}} \bigsqcup_{\eta \in \Lambda_{n-N,w-w'}^{for}} F_{(n,N),(w,w')}^{(\lambda,\eta)}(\F_q) = \bigsqcup_{\eta \in \Lambda_{n-N,w-w'}^{for}} F_{(n,N),(w,w')}^\eta(\F_q).
    \end{equation}
    We range over all possible forgettable splitting partitions $\eta \in \Lambda_{n-N, w-w'}^{for}$ to finish the proof.
\end{proof}

We now prove the main theorem of this manuscript.
\begin{proof}[Proof of Theorem \ref{theorem:main_theorem}]
    Suppose that $m_{n,q} > \max\{e^{e^e}, \log 6 + \log(p^3 + g_{E[p]}), \deg \Delta_E, 6 \log p + 2\}$. Let $\rho \in (0,1)$ be any fixed number.
    From Proposition \ref{proposition:fan_approximation}, we obtain that 
    \begin{align}
    \begin{split}
        & \; \; \; \; \# F_n(\F_q) - \sum_{w = \rho m_{n,q}}^{2 m_{n,q}} \sum_{w' = (1-\epsilon)w}^w \sum_{N = w'\mathfrak{n}}^n \# \hat{F}_{(n,N),(w,w')}(\F_q) \\
        &\leq 4 \cdot q^n \cdot \max \left((n\log q)^{-\rho \log \rho - 1 + \rho}, 3 \cdot m_{n,q}^2 \cdot \left( \frac{p}{p^2-1} \right)^{(1-\epsilon)\rho m_{n,q}}  \right) \\
        &\leq 4 \cdot q^n \cdot \max \left((n\log q)^{-\rho \log \rho - 1 + \rho}, 3 \cdot m_{n,q}^2 \cdot (n \log q)^{(1-\epsilon)\rho \log \left( \frac{p}{p^2 - 1} \right)} \right),
    \end{split}
    \end{align}
    where $\epsilon = (\log \log m_{n,q})^{-1}$.
    Letting $w$ to satisfy $\rho m_{n,q} \leq w < 2 m_{n,q}$, and $(1-\epsilon)w \leq w' \leq w$, we obtain from Proposition \ref{prop:global-local} that
    \begin{align} \label{eqn:maineq5}
        \begin{split}
            & \; \; \; \; \left| \frac{\#\{f \in \hat{F}_{(n,N),(w,w')}(\F_q) \; | \; \dim_{\mathbb{F}_p} \Sel_\pi(E^{\chi_f}) = J\}}{\# \hat{F}_{(n,N),(w,w')}(\F_q)} - PR(J) \right| \\
        &< \tilde{B}_{E,p,q} \cdot (n \log q)^{4 \epsilon \log p} \cdot \left( (n \log q)^{-m_{n,q}} + 3 \cdot (n \log q)^{(1-\epsilon)\rho \log \gamma_p } \right) \\
        &< 6 \cdot \tilde{B}_{E,p,q} \cdot (n \log q)^{(1-\epsilon)\rho \log \gamma_p + 4 \epsilon \log p}.
        \end{split}
    \end{align}
    Combine two equations to obtain
    \begin{align} \label{eqn:maineq6}
        \begin{split}
            & \; \; \; \; \left| \frac{\#\{f \in F_n(\F_q) \; | \; \dim_{\mathbb{F}_p} \Sel_\pi(E^{\chi_f}) = J\}}{\# F_n(\F_q)} - PR(J) \right| < \frac{12 \cdot m_{n,q}^2 \cdot \tilde{B}_{E,p,q}}{(n\log q)^{\alpha(p,\rho,\epsilon)}},
        \end{split}
    \end{align}
    where
    \begin{equation*}
        \alpha(p,\rho,\epsilon) := \min \begin{cases} 
        & \rho \log \rho + 1 - \rho, \\
        &-(1-\epsilon)\rho \log \left(\frac{p}{p^2-1} \right), \\
        &-(1-\epsilon) \rho \log \gamma_p + 4 \epsilon \log p. \end{cases} 
    \end{equation*}
    By substituting $\epsilon = (\log \log m_{n,q})^{-1}$, we have
    \begin{equation*}
        \tilde{B}_{E,p,q} := B_{E,p,q} + c_p \leq (B_{E,p,q} + c) \cdot p^{\max_{\chi \in \Omega_E} \text{rk}(\chi)},
    \end{equation*}
    \begin{equation*}
        \alpha(p,\rho,\epsilon) = \min \begin{cases} 
        & \rho \log \rho + 1 - \rho, \\
        &-\rho \log \left(\frac{p}{p^2-1} \right) + O \left(\frac{1}{\log \log m_{n,q}} \right), \\
        &-\rho \log \gamma_p + O \left(\frac{1}{\log \log m_{n,q}} \right).
        \end{cases}
    \end{equation*}
    Then for any small enough $\delta > 0$, there exist sufficiently large $n$ and an explicit constant $\tilde{A}_{E,p,q} := 12 \cdot (B_{E,p,q} + c) \cdot p^{\max_{\chi \in \Omega_E} \text{rk}(\chi)}$ such that
    \begin{equation}
        \left| \frac{\#\{f \in F_n(\F_q) \; | \; \dim_{\mathbb{F}_p} \Sel_\pi(E^{\chi_f}) = J\}}{\# F_n(\F_q)} - PR(J) \right| < \frac{\tilde{A}_{E,p,q}}{(n \log q)^{\alpha(p,\rho) - \delta}},
    \end{equation}
    where $\alpha(p,\rho)$ is a function obtained from $\alpha(p,\rho,\epsilon)$ by letting $m_{n,q}$ to grow arbitrarily large:
    \begin{equation*}
        \alpha(p,\rho) := \min \begin{cases} 
        & \rho \log \rho + 1 - \rho, \\
        &-\rho \log \left(\frac{p}{p^2-1} \right), \\
        &-\rho \log \gamma_p.
        \end{cases}
    \end{equation*}
    We then define $\alpha(p) := \sup_{0 < \rho < 1} \alpha(p,\rho)$ and set $A_{E,p,q} := \tilde{A}_{E,p,q} \cdot (\log q)^{-\alpha(p) + \delta}$ to obtain the statement of the main theorem.
\end{proof}

\section*{Acknowledgements}
My apologies in advance that the acknowledgement section may be longer than what one may read from other manuscripts.

This research was conducted during my compulsory military service as a research personnel in National Institute for Mathematical Sciences (NIMS) and during my graduate studies at the University of Wisconsin-Madison after returning from leave of absence. The progress I made on this project during my military service (which comprises most of the ideas covered in this manuscript) was primarily done at the end of the day after my usual working hours as a research personnel. Some corrections to the manuscript were made during my stay as a postdoctoral researcher at Max Planck Institute for Mathematics.

I would like to express my sincerest and most immense gratefulness to my PhD advisor Jordan Ellenberg for suggesting various helpful references, giving crucial and enlightening insights, and investing valuable time and effort for having regular Skype meetings during my leave of absence. I would also like to thank him for suggesting to understand the role of generalized Riemann hypothesis in computing the desired probability distributions. Without his continuous and immensely enthusiastic and patient support, this work would not have led to its fruition. 

I would like to thank Zev Klagsbrun, Aaron Landesman, Jungin Lee, Wanlin Li, Mark Shusterman, Alex Smith, and Niudun Wang for carefully explaining their works on probability distributions of families of elliptic curves, suggesting valuable references which are relevant to this paper \cite{KMR14, FLR20, Wa21, CDLL22, Sm22_01}, and for engaging in very helpful discussions. I would like to thank Jungin Lee and Zev Klagsbrun for comments on improving the exposition of this paper. I would like to thank U Jin Choi and Yongsul Won for helpful discussions on Markov chains, and many other researchers at NIMS for shaping enjoyable experiences during my service. I would like to thank Douglas Ulmer for enlightening discussions, from which an error from from the previous version of this manuscript was discovered. I would like to thank Max Planck Institute for Mathematics for providing a wonderful environment to delve into research and actively interact with other mathematicians, especially from which the discussions with Douglas Ulmer were made possible.

I would like to sincerely thank anonymous reviewer for giving extensive comments, pointed out several mathematical errors in the previous version of this draft, and suggested a number of ways to further improve the exposition of this paper. The current state of this manuscript would not have been achievable without ample support from the referee.

Lastly, I would like to sincerely thank my parents for their wholehearted continuous support throughout my military service and in the midst of the COVID-19 global pandemic. Without their unconditional love and continuous support, the experiences I had with preparing this manuscript would not have been as enjoyable.

\nocite{*}
\bibliographystyle{alpha}
\bibliography{general_draft_version3}

\end{document}